\documentclass[english]{amsart}

\usepackage{esint}
\usepackage[svgnames]{xcolor}
\usepackage[colorlinks,citecolor=red,pagebackref,hypertexnames=false,breaklinks]{hyperref}
\usepackage{pgf,tikz}
\usepackage{pdfsync}

\usepackage{dsfont}
\usepackage{url}
\usepackage[utf8]{inputenc}
\usepackage[T1]{fontenc}
\usepackage{lmodern}
\usepackage{babel}
\usepackage{mathtools}
\usepackage{amssymb}
\usepackage{lipsum}
\usepackage{mathrsfs}
\usepackage{color}
\usepackage{marginnote}

\newtheorem{theorem}{Theorem}[section]
\newtheorem{proposition}{Proposition}[section]
\newtheorem{lemma}{Lemma}[section]

\newtheorem{corollary}{Corollary}[section]

\newtheorem{remark}{Remark}[section]

\numberwithin{equation}{section}

\title[determination of the conductivity at the boundary]{Comments on the determination of the conductivity at the boundary from the Dirichlet-to-Neumann map}

\author[Mourad Choulli]{Mourad Choulli}
\address{Universit\'e de Lorraine}
\email{mourad.choulli@univ-lorraine.fr}

\thanks{The author is supported by the grant ANR-17-CE40-0029 of the French National Research Agency ANR (project MultiOnde). }

\date{}

\begin{document}

\begin{abstract}
We revisit the stability issue of determining the conductivity at the boundary from the corresponding Dirichlet-to-Neumann map. We discuss both the method based on singular solutions and the one built on the localized oscillating solutions. Our primary objective  is not establishing new results on this subject even if the present work contains some new ones. We mainly clarify some points in the existing proofs and make some comments. We also derive some consequences of the stability inequality of the determination of the conductivity at the boundary from the Dirichlet-to-Neumann map.
\end{abstract}



\maketitle

\tableofcontents

\section{Introduction}\label{section1}

Throughout this text all functions we consider are real-valued. Let $\Omega$ be a Lipschitz bounded domain of $\mathbb{R}^n$ ($n\ge 3$), with boundary $\Gamma$, and set
\[
C_+(\overline{\Omega})=\{ \sigma \in C(\overline{\Omega});\; \sigma >0\; \mbox{in}\; \overline{\Omega}\}.
\]
We consider, where $\sigma \in C_+(\overline{\Omega})$, the symmetric bounded and coercive bilinear form
\[
\mathfrak{a}_\sigma (u,v)=(\sigma\nabla u|\nabla v)_2,\quad u,v\in H_0^1(\Omega),
\]
where $(\cdot |\cdot)_2$ is the usual scalar product  of $L^2(\Omega )$.

Let $F\in H^{-1}(\Omega)$. There exists, according to Lax-Milgram's lemma, a unique $v_\sigma (F)\in H_0^1(\Omega)$ so that
\[
\mathfrak{a}_\sigma (v_\sigma (F),w)=\langle F,w\rangle_1,\quad  w\in H^1(\Omega),
\]
where $\langle \cdot ,\cdot \rangle_1$ is the duality pairing between $H_0^1(\Omega)$ and its dual $H^{-1}(\Omega)$.

Denote by $\gamma_0$ the bounded trace operator from $H^1(\Omega )$ onto $H^{1/2}(\Gamma )$ defined by 
\[
\gamma_0w =w_{|\Gamma},\quad w\in C^\infty (\overline{\Omega }).
\]
For simplicity convenience, $\gamma_0w$ is  denoted in the sequel by $w_{|\Gamma}$.

For $h\in H^{1/2}(\Gamma)$, let $\mathcal{E}h$ denotes the unique element of $H^1(\Omega)$ satisfying $\mathcal{E}h_{|\Gamma}=h$ and
\[
\|\mathcal{E}h\|_{H^1(\Omega)}=\|h\|_{H^{1/2}(\Gamma)}.
\]
Pick $g\in H^{1/2}(\Gamma)$. It is then not difficult to check that 
\[
u_\sigma(g)=\mathcal{E}g+v_\sigma (-\mbox{div}(\sigma \nabla \mathcal{E}g))\in H^1(\Omega)
\]
 is the unique solution of the BVP

\[
\left\{
\begin{array}{l}
\mbox{div}(\sigma \nabla u)=0\quad \mbox{in}\; \Omega,
\\
u_{|\Gamma}=g.
\end{array}
\right.
\]
Furthermore,
\begin{equation}\label{i1}
\|u_\sigma(g)\|_{H^1(\Omega )}\le \varkappa \|g\|_{H^{1/2}(\Gamma)},
\end{equation}
where $\varkappa =\varkappa (n,\Omega ,\min \sigma) >0$ is a constant. 

We define the Dirichlet-to-Neumann map, associated to $\sigma \in C_+(\overline{\Omega})$, by
\begin{align*}
\Lambda_\sigma :g\in H^{1/2}(\Gamma )\mapsto &\Lambda_\sigma (g)\in H^{-1/2}(\Gamma): 
\\
&\langle\Lambda_\sigma (g),h\rangle_{1/2}=\mathfrak{a}_\sigma(u_\sigma(g),\mathcal{E}h),\quad h\in H^{1/2}(\Gamma),
\end{align*}
where $\langle \cdot ,\cdot\rangle_{1/2}$ is the duality pairing between $H^{1/2}(\Gamma)$ and its dual $H^{-1/2}(\Gamma)$. We remark that we have, according to \cite[Lemma 2.2 in page 131]{Ka}, 
\[
\Lambda_\sigma(g)=\sigma\partial_\nu u_\sigma(g), 
\]
where $\partial_\nu$ denotes the derivative along the unit normal exterior vector field to $\Gamma$. In light of \eqref{i1}, we obtain
\[
|\langle\Lambda_\sigma (g),h\rangle_{1/2}|\le \varkappa \|\sigma\|_{C(\overline{\Omega})}\|g\|_{H^{1/2}(\Gamma)}\|h\|_{H^{1/2}(\Gamma)},\quad h\in H^{1/2}(\Gamma).
\]
Whence  $\Lambda_\sigma \in \mathscr{B}(H^{1/2}(\Gamma ),H^{-1/2}(\Gamma))$.

For notational convenience  the natural norm  of $\mathscr{B}(H^{1/2}(\Gamma ),H^{-1/2}(\Gamma))$ will simply denoted in the rest of this text by $\|\cdot \|$.

Fix $0<\alpha\le1$, $\kappa >1$ and let
\[
\Sigma = \left\{ \sigma \in C^{1,\alpha}(\overline{\Omega});\; \kappa^{-1}\le \sigma,  \;  \|\sigma\|_{C^{1,\alpha}(\overline{\Omega})}\le \kappa\right\}.
\]

\begin{theorem}\label{theorem-i1}
If $\Omega$ is of class $C^{1,1}$ then, for all $\sigma_1,\sigma_2\in \Sigma$, we have
\begin{align}
& \|\sigma_1-\sigma_2 \|_{C(\Gamma)}\le C\|\Lambda_{\sigma_1}- \Lambda_{\sigma_2}\|,\label{thm-i1.1}
\\
& \|\partial_\nu (\sigma_1-\sigma_2) \|_{C(\Gamma)}\le C\|\Lambda_{\sigma_1}- \Lambda_{\sigma_2}\|^{\alpha/(\alpha+1)}, \label{thm-i1.2}
\end{align}
where $C=C(n,\Omega ,\kappa,\alpha )>0$ is a constant.
\end{theorem}

As $\alpha\in ]0,1] \mapsto \alpha/(1+\alpha)$ is strictly increasing, the best possible exponent in \eqref{thm-i1.2} is obtained when $\alpha=1$ and it is equal to $1/2$.

In light of the interpolation inequality in \cite[Lemma 3.2 in page 264]{Al90} (in which we substitute $\tilde{\nu}$ by $\nu$), we deduce from Theorem \ref{theorem-i1} the following corollary.

\begin{corollary}\label{corollary-i1}
Assume that  $\Omega$ is of class $C^{1,1}$.  Then, for all $\sigma_1,\sigma_2\in \Sigma$, we have
\begin{align*}
& \|\sigma_1-\sigma_2 \|_{C(\Gamma)}\le C\|\Lambda_{\sigma_1}- \Lambda_{\sigma_2}\|,
\\
& \|\nabla (\sigma_1-\sigma_2) \|_{C(\Gamma)}\le C\|\Lambda_{\sigma_1}- \Lambda_{\sigma_2}\|^{\alpha/(\alpha+1)}, 
\end{align*}
where $C=C(n,\Omega ,\kappa,\alpha )>0$ is a constant.
\end{corollary}

Theorem \ref{theorem-i1} was established by Alessandrini \cite{Al90} using singular solutions. 
An earlier result by Kohn and Vogelius \cite{KV84} gives a uniqueness of the conductivity and its normal derivative at the boundary. We also mention the works of Brown \cite{Br} and Nakamura and Tanuma \cite{NT} for reconstruction formulas (for derivatives of arbitrary order in \cite{NT}). The main idea introduced in \cite{KV84} consists in constructing oscillating solutions localized at a boundary point. We revisit this construction in the last section to show that it gives also  stability inequalities similar to that in Theorem \ref{theorem-i1}. The only difference is that we obtain, instead of $\alpha/(\alpha+1)$ in \eqref{thm-i1.2}, the exponent $\alpha/[2(\alpha+1)]$. We initially expected to get by this method the same exponent as in \eqref{thm-i1.2} but actually we do not succeed to modify our analysis to prove it.

A variant of Theorem \ref{theorem-i1} was proven by Sylvester and Uhlamnn \cite{SU88} by using tools from microlocal analysis. They showed that the informations on conductivity and its normal derivative at the boundary are contained in the two principal terms of the full symbol of $\Lambda_\sigma$, considered as a pseudo-differential operator of order one. These informations are extracted by  using again oscillating solutions localized at a boundary point. The results in \cite{SU88} yield a H\"older stability for the normal derivative at the boundary, with an exponent $\gamma$ satisfying $0<\gamma< 1/(n+1)$ (see for instance \cite[Theorem 4.2 in page 6]{Uh09}).

We define, for fixed  $\dot{\kappa}>1$ and $p>n$,
\[
\dot{\Sigma}=\left\{\sigma \in W^{2,p}(\Omega);\; \dot{\kappa}^{-1} \le \sigma ,\; \|\sigma\|_{W^{2,p}(\Omega)}\le \dot{\kappa}\right\}.
\]

In light of \cite[Corollaire 1.2 in page 30]{Ch13}, we can assert that $W^{2,p}(\Omega)$ is continuously embedded in $C^{1,\alpha}(\overline{\Omega})$, when $\alpha=1-n/p$. Therefore  we have obviously $\dot{\Sigma}\subset \Sigma$, with $\alpha=1-n/p$ and some constant $\kappa =\kappa (n,\Omega ,p, \dot{\kappa})>1$.

For $0<c<1$, define $\Phi_c$ by 
\[
\Phi_c(\varrho)=\left\{
\begin{array}{ll}
0 &\mbox{if}\; \varrho=0,
\\
|\ln \varrho|^{-2/(n+2)}\quad &\mbox{if}\; \varrho \in (0,c),
\\
\varrho &\mbox{if}\; \varrho \in [c,\infty).
\end{array}
\right.
\]

\begin{theorem}\label{theorem-i2}
Suppose that  $\Omega$ is of class $C^{1,1}$. Then, for all $\sigma_1,\sigma_2\in \dot{\Sigma}$, we have
\begin{equation}\label{thm-i2.1}
\|\sigma_1-\sigma_2\|_{H^1(\Omega)}\le C\Phi_c \left(\|\Lambda_{\sigma_1}-\Lambda_{\sigma_2}\|\right),
\end{equation}
where $c=c(n,\Omega ,p,\dot{\kappa})>0$ and $C=C(n,\Omega ,p,\dot{\kappa})>0$ are constants.
\end{theorem}

Alessandrini \cite{Al88} established a similar inequality to that in Theorem \ref{theorem-i2} for $H^{s+2}$, when $s>n/2$, conductivities with a logarithmic modulus of continuity of indefinite exponent $-\gamma$, $0<\gamma=\gamma (n,s)<1$.

The interior stability inequality is achieved by performing the usual Liouville's transform in order to reduce the inverse conductivity problem into the problem of recovering the potential $q$ in $-\Delta +q$, from the corresponding Dirichlet-to-Neumann map. When $\sigma\in  \dot{\Sigma}$, the associated potential $q_\sigma$ belongs to $L^p(\Omega)$. It is not necessarily bounded but still a function. We point out that the regularity on the conductivity can be relaxed. The case of $C^{1,\alpha}$ conductivities was considered by Caro, Garcia and Reyes \cite{CGR}, in which $q_\sigma$ is only a distribution (see the exact statement in \cite[Theorem 1.1, page 470]{CGR}). In that case the proof is more intricate. The analysis in \cite{CGR} follows the one introduced by Haberman and Tataru in \cite{HT} to prove uniqueness for $C^1$ (or Lipschitz and close to a constant) conductivities. The case of Lipschitz conductivities was conjectured by Uhlmann and proved by Caro and Rogers in \cite{CR}. We point out that there is only few results in the case of less regular conductivities and partial boundary data. We refer to the recent paper by Krupchyk and Uhlmann \cite{KU} where uniqueness results are established for conductivities with only $3/2$ derivatives.

The first uniqueness result of determining piecewise real-analytic conductivities from the corresponding Dirichlet-to-Neumann map was obtained in the earlier work by Kohn and Vogelius \cite{KV85}.

We define the map $\Lambda$ by
\[
\Lambda :\sigma \in C_+(\overline{\Omega})\mapsto \Lambda_\sigma \in \mathscr{B}(H^{1/2}(\Gamma ),H^{-1/2}(\Gamma)).
\]
We leave to the reader to check that $\Lambda$ is continuous.

\begin{theorem}\label{theorem-i3}
Assume that  $\Omega$ is of class $C^{1,1}$. Then there exists  a dense subset $\mathscr{D}$ of $C_+(\overline{\Omega})$, endowed with the topology of $C(\overline{\Omega})$, so that $\Lambda_{|\mathscr{D}}$ is injective.
\end{theorem}

It is worth noticing that $\Lambda$ can be extended to continuous map from $L^\infty_+(\Omega)$ into $\mathscr{B}(H^{1/2}(\Gamma ),H^{-1/2}(\Gamma))$, where
\[
L^\infty_+(\Omega)=\{ \sigma\in L^\infty(\Omega);\; \mathrm{essinf}\sigma >0\}.
\]
We know from Lusin's theorem that every function of $L^\infty(\Omega)$ can be approximated pointwise (in the almost everywhere sense) by a sequence of $C(\overline{\Omega})$. Unfortunately, this is not sufficient to extend Theorem \ref{theorem-i3} to bounded conductivities.

There is a tremendous amount of literature devoted to the inverse conductivity problem. We will not discuss in details this literature. We refer to the nice review paper by Uhlmann \cite{Uh09} on the inverse conductivity problem and the related topics, starting from the pioneer paper by Calder\'on \cite{Ca}. This review paper contains also the most significative results for the two dimensional case whose treatment uses tools from complex analysis. For sake of clarity, we do not comment in the present work the two dimensional case.

We close this introduction by remarking that it appears from our analysis that $C^{1,1}$ regularity of the domain seems to be the best possible one.

\section{Stability at the boundary using singular solutions}\label{section2}

We shall use in the sequel the following extension theorem.

\begin{theorem}\label{exth}
Assume that $\Omega$ is of class $C^{1,1}$ and $0\le \beta \le 1$. Then there exists $\mathscr{E}_\beta\in \mathscr{B}(C^{1,\beta}(\overline{\Omega}),C^{1,\beta}(\mathbb{R}^n))$, preserving positivity, so that $\mathscr{E}_\beta f_{|\Omega}=f$, for any $f\in C^{1,\beta}(\overline{\Omega})$. Furthermore,
\[
\|\mathscr{E}_\beta\|_{ \mathscr{B}(C^{1,\beta}(\overline{\Omega}),C^{1,\beta}(\mathbb{R}^n))}\le C,
\]
for some constant $C=C(n,\Omega,\beta)>0$.
\end{theorem}

\begin{proof}
This theorem is more or less known. One can recover for instance the proof by modifying slightly that of \cite[Theorem 1.16, page 23]{Ch13}.
\end{proof}

Define $\Sigma^e = \mathscr{E}_\alpha (\Sigma)$. In light of Theorem \ref{exth} we can assert that
\[
\Sigma^e\subset\left\{ \sigma\in C^{1,\alpha}(\mathbb{R}^n);\; \kappa^{-1}\le \sigma, \; \|\sigma\|_{C^{1,\alpha}(\mathbb{R}^n)}\le \kappa^e\right\},
\]
where $\kappa^e=C\kappa$, with $C$ as in the preceding theorem in which $\beta$ is substituted by $\alpha$.

Consider, for each $\sigma \in \Sigma$, the operator $L_\sigma$ acting as follows
\[
L_\sigma u =\mbox{div}(\sigma \nabla u),\quad u\in C^2(\Omega),
\]
 and set
\[
\mathscr{S}_\sigma =\left\{u\in C^2(\overline{\Omega});\; L_\sigma u=0\right\},\quad \sigma \in \Sigma_\kappa.
\]

\subsection{Proof of \eqref{thm-i1.1} of Theorem \ref{theorem-i1}}\label{subsection2.1}

Denote by $\mathcal{L}_\sigma$, $\sigma \in \Sigma$, the operator that acts as follows
\[
\mathcal{L}_\sigma u=\Delta u+\nabla \ln \sigma \cdot \nabla u, \quad u\in C^2(\Omega),
\]
We present a proof based on the singularities of the fundamental solution of the operator $\mathcal{L}_\sigma$, obtained by Levi's parametrix method.

Let $\Omega_0 \Supset \Omega$ be a Lipschitz domain and set $\Omega_+=\Omega$ and $\Omega_-=\Omega_0 \setminus \overline{\Omega}$. Since $\Omega_-$ and $\Omega_+$ are both Lipschitz domains, they possess the uniform interior cone property. Therefore, there exists $R>0$ and $\theta\in  ]0,\pi/2[$ so that, for each $x_0\in \Gamma$, we find $\xi_\pm =\xi_\pm (x_0)\in \mathbb{S}^{n-1}$ with the property that
\[
\mathscr{C}_\pm(x_0)=\{x\in \mathbb{R}^n;\; 0<|x-x_0|<R,\;  (x-x_0)\cdot \xi_\pm >|x-x_0|\cos \theta\}\subset \Omega_\pm.
\]
The following fact will be useful in the sequel : if $x_\rho=x_0+\rho \xi_\pm$, with $0<\rho < R/2$, then
\[
\mbox{dist}(x_\rho,\partial \mathscr{C}_\pm(x_0))=\rho\sin \theta . 
\]
Pick $x_0\in \Gamma $. Let $R>0$ and $\xi_\pm =\xi_\pm(x_0)$ be as in the definition of the interior cone property.  We use in what follows the notations
\[
x_\delta =x_0+(\delta/2) \xi_+,\; y_\delta =x_0+(\delta/2) \xi_-,\quad 0<\delta <R/2. 
\]
Of course $x_\delta$ and $y_\delta$ depend on $x_0$.

We denote by $H$ the usual fundamental solution of the Laplace operator:
\[
H(x,y)= \left[(n-2)|\mathbb{S}^{n-1}|\right]^{-1}|x-y|^{2-n}, \quad x,y\in \mathbb{R}^n,\; x\ne y.
\]
Then straightforward computations show that
\[
|\nabla H(x,y)|^2= |\mathbb{S}^{n-1}|^{-2}|x-y|^{2-2n}, \quad x,y\in \mathbb{R}^n,\; x\ne y.
\]

The following result follows readily from \cite[Theorem 5, page 282]{Kal} (see also \cite[Theorem A.7, page 265]{Chou}) applied to the operator $\mathcal{L}_\sigma$.

\begin{theorem}\label{theorem1}
For any $\sigma\in \Sigma$ and $y\in \Omega_0\setminus \overline{\Omega}$, there exists $u_{\sigma} ^y\in \mathscr{S}_\sigma$ so that
\begin{align}
&|u_{\sigma} ^y(x)-H(x,y)|\le C|x-y|^{2-n+\alpha},\quad x\in \overline{\Omega},\label{thm1.1}
\\
&|\nabla u_{\sigma} ^y(x)-\nabla H(x,y)|\le C|x-y|^{1-n+\alpha},\quad x\in \overline{\Omega},\label{thm1.2}
\end{align}
where  $C=C(n,\Omega, ,\alpha ,\kappa)>0$ is a constant.
\end{theorem}

The preceding result is obtained from \cite[Theorem 5, page 282]{Kal} with a $C^2$-smooth domain $\Omega_1$ satisfying $\Omega_1 \Supset \Omega_0\Supset \Omega$. In that case $\mathscr{E}_\alpha\sigma$ gives an extension of $\sigma$ in $\Omega_1$ with the properties required in \cite[Theorem 5, page 282]{Kal}.

The following lemma can be deduced easily from the preceding theorem.

\begin{lemma}\label{lemma3.1}
There exist three constants $C=C(n,\Omega, ,\alpha ,\kappa)>0$, $0<\mathfrak{r}=\mathfrak{r}(n,\Omega, ,\alpha ,\kappa)\le R$ and $0<\delta_0=(n,\Omega, ,\alpha ,\kappa)\le \mathfrak{r}/2$ so that
\begin{equation}\label{lem3.1}
\nabla u_1(x)\cdot \nabla u_2(x)\ge C|x-y_\delta |^{2-2n},\quad x\in B(x_0,\mathfrak{r})\cap \Omega,\; 0<\delta \le \delta_0,
\end{equation}
where $u_j=u_{\sigma_j}^{y_\delta}$ with $\sigma_j\in \Sigma$, $j=1,2$, is as in Theorem \ref{theorem1}.
\end{lemma}

It is worth noticing that this lemma says that $\nabla u_1(x)\cdot \nabla u_2(x)$ behaves like $|\nabla H(x,y_\delta)|^2$, locally near $x_0$.

In the rest of this subsection, $C=C(n,\Omega, ,\alpha ,\kappa)>0$ will denote a generic constant. Also, the constants $\mathfrak{r}$ and $\delta_0$ are the same as in Lemma \ref{lemma3.1}.

Pick $\sigma_j\in \Sigma$, $j=1,2$, and set $\sigma =\sigma_1-\sigma_2$. Fix $x_0\in \Gamma $ so that $|\sigma(x_0)|=\|\sigma\|_{C(\Gamma)}$ and, without loss of generality, we may assume that $|\sigma(x_0)|=\sigma(x_0)$.

Let $u_j=u_{\sigma_j}^{y_\delta}\in \mathscr{S}_{\sigma_j}$, $j=1,2$, be given by Theorem \ref{theorem1} and set
\[
v_j=u_j-\int_\Omega u_jdx \; (\in \mathscr{S}_{\sigma_j}).
\]
Since
\[
\|\sigma\|_{C(\Gamma)}= \sigma(x_0)\le \sigma (x)+2\kappa |x-x_0|^\alpha,\quad x\in \Omega,
\]
we get
\begin{align}
\|\sigma\|_{C(\Gamma)}\int_{B(x_0,\mathfrak{r})\cap \Omega} \nabla u_1\cdot \nabla u_2dx\le &\int_{B(x_0,\mathfrak{r})\cap \Omega} \sigma \nabla u_1\cdot \nabla u_2dx \label{a1}
\\
&+2\kappa\int_{B(x_0,\mathfrak{r})\cap \Omega}  |x-x_0|^\alpha\nabla u_1\cdot \nabla u_2dx.\nonumber
\end{align}

Hereafter, $0<\delta \le \delta_0$. Using that 
\[
B(x_\delta ,\delta \sin\theta /2)\subset B(x_0,\delta)\cap \mathscr{C}_+(x_0)\subset B(x_0,\mathfrak{r})\cap \Omega,
\]
we get from \eqref{lem3.1}
\begin{equation}\label{a2}
\int_{B(x_0,\mathfrak{r})\cap \Omega} \nabla u_1\cdot \nabla u_2dx\ge |B(x_\delta ,\delta \sin\theta /2)| (3\delta/2)^{2-2n}\ge C\delta^{2-n}.
\end{equation}
Taking into account that
\[
|x-y_\delta|\ge \mathfrak{r}/2,\quad x\in \Omega \setminus B(x_0,\mathfrak{r}),
\]
and
\[
|\nabla u_j|\le C|x-y_\delta|^{2-2n},\quad j=1,2,
\]
we obtain
\[
\int_{\Omega \setminus B(x_0,\mathfrak{r})}\sigma \nabla u_1\cdot\nabla u_2\le C.
\]
Therefore
\begin{equation}\label{a3}
\int_{B(x_0,\mathfrak{r})\cap \Omega} \sigma \nabla u_1\cdot \nabla u_2dx\le \int_\Omega \sigma \nabla u_1\cdot \nabla u_2dx +C.
\end{equation}

Let $\mathrm{R}$ (independent of $x_0$) sufficiently large in such a way that
\[
\Omega \subset B(y_\delta ,\mathrm{R})\setminus B(y_\delta, \delta\sin \theta /2).
\]
In that case, we have
\[
\int_{B(x_0,\mathfrak{r})\cap \Omega}  |x-x_0|^\alpha\nabla u_1\cdot \nabla u_2dx\le |\mathbb{S}^{n-1}|\int_{\delta\sin \theta /2}^{\mathrm{R}}(\delta/2+r)^\alpha r^{1-n}dr.
\]
In consequence,
\begin{equation}\label{a4}
\int_{B(x_0,\mathfrak{r})\cap \Omega}  |x-x_0|^\alpha\nabla u_1\cdot \nabla u_2dx\le C\delta^{2-n+\alpha}.
\end{equation}
Inequalities \eqref{a2}, \eqref{a3} and \eqref{a4} in \eqref{a1} give

\begin{equation}\label{3.3}
C\|\sigma\|_{C(\Gamma)}\le \delta^{n-2}\int_\Omega \sigma \nabla v_1\cdot \nabla v_2dx+\delta^\alpha .
\end{equation}

Let
\[
V=\left\{u\in H^1(\Omega);\; \int_\Omega u(x)dx=0\right\}.
\]
Then Poincar\'e's inequality shows the map $w\mapsto \|\nabla w\|_{L^2(\Omega)}$ defines a norm on $V$ equivalent to the usual $H^1$-norm.

We get in a straightforward manner, by using inequality \eqref{thm1.2}, 
\[
\|\nabla v_j\|_{L^2(\Omega )}\le C\delta ^{(2-n)/2},\quad j=1,2.
\]
Combined with the continuity of the trace operator $w\in H^1(\Omega )\mapsto w_{|\Gamma}\in H^{1/2}(\Gamma)$, these inequalities imply 
\begin{equation}\label{3.4}
\|v_j\|_{H^{1/2}(\Gamma)}\le C\delta ^{(2-n)/2},\quad j=1,2.
\end{equation}

Now according to a well known identity (e.g. \cite[formula (3.4), page 265]{Al90}),  we have
\[
\int_\Omega \sigma \nabla v_1\cdot \nabla v_2dx=\langle (\Lambda_1-\Lambda_2)v_1,v_2\rangle_{1/2}.
\]
Here and henceforward, $\Lambda_j=\Lambda_{\sigma_j}$, $j=1,2$. Inequality \eqref{3.4}  then yields 
\begin{equation}\label{3.5}
\int_\Omega \sigma \nabla v_1\cdot \nabla v_2\le C\delta ^{2-n}\|\Lambda_1-\Lambda_2\|.
\end{equation}
Putting together \eqref{3.3} and \eqref{3.5} in order to get
\[
C\|\sigma\|_{C(\Gamma)}\le \|\Lambda_1-\Lambda_2\|+\delta^{\alpha}.
\]
We obtain by passing to the limit, when $\delta$ tends to $0$,
\[
\|\sigma\|_{C(\Gamma)}\le C \|\Lambda_1-\Lambda_2\|.
\]
That is we proved \eqref{thm-i1.1}.

\subsection{Proof of \eqref{thm-i1.2} of Theorem \ref{theorem-i1}}\label{subsection2.2}

We shall need the following proposition. We provide its proof in Appendix \ref{appendixA}. We use hereafter the notation
\[
\Omega_r=\{x\in \Omega;\; \mathrm{dist}(x,\Gamma)\le r\},\quad r>0.
\]

\begin{proposition}\label{gproposition}
Suppose that $\Omega$ is of class $C^{1,1}$. There exists $\dot{\varrho}=\dot{\varrho}(n,\Omega)$ so that  we have :
\\
$\mathrm{(i)}$ For any $x\in \Omega_{\dot{\varrho}}$, there exists a unique $\mathfrak{p}(x)\in \Gamma$ such that
\[
|x-\mathfrak{p}(x)|=\mathrm{dist}(x,\Gamma),\; x= \mathfrak{p}(x)-|x-\mathfrak{p}(x)|\nu(\mathfrak{p}(x)).
\]
(ii) If $x\in \Omega_{\dot{\varrho}}$ then $x_t=\mathfrak{p}(x)-t|x-\mathfrak{p}(x)|\nu(\mathfrak{p}(x))\in \Omega_{\dot{\varrho}}$, $t\in ]0,1]$, and
\[
\mathfrak{p}(x_t)=\mathfrak{p}(x),\; |x_t-\mathfrak{p}(x)|=t\mathrm{dist}(x,\Gamma).
\]
\end{proposition}

In the rest of this text, we keep the notations $\dot{\varrho}$,  $\Omega_{\dot{\varrho}}$ and $\mathfrak{p}(x)$, $x\in \Omega_{\dot{\varrho}}$, as they are defined in Proposition \ref{gproposition}. We will also use the following semi-norm
 \[
 [\nabla f]_\alpha=\sup\left\{\frac{|\nabla f(x)-\nabla f(y)|}{|x-y|^\alpha};\; x,y\in \overline{\Omega},\; x\ne y\right\},\quad f\in C^{1,\alpha}(\overline{\Omega}).
 \]

\begin{lemma}\label{lemma2}
Assume that $\Omega$ is of class $C^{1,1}$ and fix $\varkappa>0$. Let $f\in C^{1,\alpha}(\overline{\Omega})$ satisfying, for some $x_0\in \Gamma$,
\[
[\nabla f]_\alpha \le \varkappa\quad \mathrm{and}\quad  -\partial_\nu f(x_0)=|\partial_\nu f(x_0)|>0.  
\]
Then,  we have
\begin{equation}\label{lem2.1}
|\partial_\nu f(x_0)|\mathrm{dist}(x,\Gamma)\le f(x)-f(\mathfrak{p}(x))+3\varkappa |x-x_0|^{1+\alpha},\quad x\in \Omega_{\dot{\varrho}}.
\end{equation}

\end{lemma}
\begin{proof}
 
Let $x\in \Omega_{\dot{\varrho}}$ and set $\tilde{x}=\mathfrak{p}(x)$. Then
\[
-\partial_\nu f(\tilde{x})\ge -\partial_\nu f(x_0)-|\partial_\nu f(\tilde{x})-\partial_\nu f(x_0)|,
\]
and hence 
\[
-\partial_\nu f(\tilde{x})\ge -\partial_\nu f(x_0)-\varkappa |\tilde{x}-x_0|^\alpha .
\]
But 
\[
|x_0-\tilde{x}|\le |x_0-x|+|x-\tilde{x}|\le 2|x-x_0|.
\]
Therefore
\begin{equation}\label{lem2.2}
-\partial_\nu f(\tilde{x})\ge -\partial_\nu f(x_0)-2\varkappa |\tilde{x}-x_0|^\alpha .
\end{equation}
Since
\begin{align*}
f(x)-f(\tilde{x})&=f(\tilde{x}-|x-\tilde{x}|\nu (\tilde{x}))-f(\tilde{x})
\\
&=-|x-\tilde{x}|\int_0^1\nabla f(\tilde{x}-t|x-\tilde{x}|\nu (\tilde{x}))\cdot\nu (\tilde{x})dt,
\end{align*}
we obtain
\begin{align*}
f(x)-f(\tilde{x})&=-\partial_\nu f(\tilde{x})|x-\tilde{x}|
\\
&-|x-\tilde{x}|\int_0^1[\nabla f(\tilde{x}-t|x-\tilde{x}|\nu (\tilde{x}))-\nabla f(\tilde{x})]\cdot\nu (\tilde{x})dt.
\end{align*}
In light of \eqref{lem2.2} this identity yields
\begin{align*}
f(x)-f(\tilde{x})+2\varkappa &|x-x_0|^{1+\alpha} \ge |\partial_\nu f(x_0)||x-\tilde{x}|
\\
&-|x-\tilde{x}|\int_0^1[\nabla f(\tilde{x}-t|x-\tilde{x}|\nu (\tilde{x}))-\nabla f(\tilde{x})]\cdot\nu (\tilde{x})dt.
\end{align*}
Whence 
\[
f(x)-f(\tilde{x})+3\varkappa |x-x_0|^{1+\alpha}\ge |\partial_\nu f(x_0)||x-\tilde{x}|.
\]
This is the expected inequality because $|x-\tilde{x}|=\mathrm{dist}(x,\Gamma)$.
\end{proof}

We observe that the singularities of the solutions we used in the preceding subsection depend on the dimension. Therefore this is not sufficient to establish \eqref{thm-i1.2} of Theorem \ref{theorem-i1} when the dimension is three as we will explain now. Fix then $x_0\in \Gamma$ so that $|\partial_\nu \sigma (x_0)|=\|\partial_\nu \sigma\|_{C(\Gamma)}$. Without loss of generality, we may assume that $|\partial_\nu \sigma (x_0)|=-\partial_\nu \sigma (x_0)$. We then apply inequality \eqref{lem2.1} in Lemma \ref{lemma2} in order to get
\begin{equation}\label{3.6}
\|\partial_\nu \sigma\|_{C(\Gamma)}\mathrm{dist}(x,\Gamma )\le \|\sigma\|_{C(\Gamma)}+\sigma(x)+6\kappa |x-x_0|^{1+\alpha},\quad x\in \Omega \cap B(x_0,\rho),
\end{equation}
where $0< \rho\le \min(\delta_0,\dot{\varrho})$.

Let $R$ be as it appears in the definition of the uniform interior cone property. Then reducing $R$ if necessary, we may assume that $R\le \dot{\varrho}$ . 

In the sequel the notations are those of the preceding subsection. Recall that 
\[
B(x_\delta ,\delta \sin\theta /2)\subset \mathscr{C}_+(x_0)\cap B(x_0,\delta) \subset B(x_0,\delta)\cap \Omega. 
\]
Therefore, noting that $\mathrm{dist}(x,\Gamma)\ge \delta \sin \theta /4$, if $x\in B(x_\delta ,\delta \sin\theta /4)$, we get 
\[
\int_{B(x_0,\delta)\cap \Omega}\mathrm{dist}(x,\Gamma)\nabla v_1\cdot \nabla v_2dx\ge C\delta^{n-3}.
\]
This together with \eqref{3.6}, \eqref{thm-i1.1} of Theorem \ref{theorem-i1}, \eqref{a3} and \eqref{3.5}  yield
\begin{equation}\label{3.7}
C\|\partial_\nu \sigma\|_{C(\Gamma)}\le \delta^{-1}\|\Lambda_1-\Lambda_2\|+\delta^{n-3}(1+\delta ^{-(n-3-\alpha)_+}),
\end{equation}
where $t_+=\max(t,0)$, $t\in \mathbb{R}$. This inequality allows us to prove \eqref{thm-i1.2} of Theorem \ref{theorem-i1} but only when $n\ge 4$. To overcome this restriction we need singular solutions with singularities of arbitrary order. The construction of such singular solutions is due to Alessandrini \cite[Lemma 3.1 in page 264]{Al90} in the case of $W^{1,p}(\Omega)$, $p>n$, conductivities (note that $C^{1,\alpha}(\overline{\Omega})$ is continuously embedded in $W^{1,p}(\Omega)$, for any $p>n$). 

\begin{theorem}\label{thm-Al90}
Let $\sigma_j\in \Sigma$, $j=1,2$, and $\ell \ge 1$ an integer. Then there exists $u_j\in W^{2,p}(\Omega )$ satisfying $L_{\sigma_j}u_j=0$ in $\Omega$, $j=1,2$, and
\begin{align*}
&|\nabla u_j(x)|\le C|x-y_\delta|^{1-(n+\ell)},\quad x\in \overline{\Omega},\; j=1,2,
\\
&\nabla u_1(x)\cdot \nabla u_2(x)\ge C|x-y_\delta|^{2-2(n+\ell)},\quad x\in \overline{\Omega},
\end{align*}
where $C=C(n,\Omega ,\alpha ,\kappa,\ell)$ is a generic constant.
\end{theorem}

By taking instead the solutions given by Lemma \ref{lemma3.1} those in Theorem \ref{thm-Al90}, we  can proceed similarly as above to get in place of \eqref{3.7} the following inequality 
\[
C\|\partial_\nu \sigma\|_{C(\Gamma)}\le \delta^{-1}\|\Lambda_1-\Lambda_2\|+\delta^{n-3+2\ell }(1+\delta ^{-(n+2\ell -3-\alpha)_+}).
\]
We fix then $\ell$ sufficiently large in such a way that $n-3+2\ell\ge 1$ in order to get
\[
C\|\partial_\nu \sigma\|_{C(\Gamma)}\le \delta^{-1}\|\Lambda_1-\Lambda_2\|+\delta ^\alpha,
\]
from which we derive \eqref{thm-i1.2} of Theorem \ref{theorem-i1} in a straightforward manner.

\section{Proof of Theorems \ref{theorem-i2} and \ref{theorem-i3}}\label{section3}

\subsection{Proof of Theorems \ref{theorem-i2}}\label{subsection3.1}

Let $\sigma \in \dot{\Sigma}$. Then the multiplication by $\sigma^{\pm 1/2}$ as an  operator, denoted again by $\sigma^{\pm 1/2}$, acting on $H^{1/2}(\Gamma)$ is bounded with
\begin{equation}\label{4.1}
\|\sigma^{\pm 1/2}\|_{\mathscr{B}(H^{1/2}(\Gamma))}\le C\|\sigma^{\pm 1/2}\|_{C^1(\Gamma)},
\end{equation}
where $C=C(n,\Omega)$ is a constant.

Recall that if $(\sigma^{\pm 1/2})^\ast$ is the adjoint of $\sigma^{\pm 1/2}$, acting as a bounded operator on $\mathscr{B}(H^{-1/2}(\Gamma))$, then
\[
\|(\sigma^{\pm 1/2})^\ast\|_{\mathscr{B}(H^{-1/2}(\Gamma))}=\|\sigma^{\pm 1/2}\|_{\mathscr{B}(H^{1/2}(\Gamma))}.
\]
This and \eqref{4.1} yields
\begin{equation}\label{4.2}
\|(\sigma^{\pm 1/2})^\ast\|_{\mathscr{B}(H^{-1/2}(\Gamma))}\le C\|\sigma^{\pm 1/2}\|_{C^1(\Gamma)},
\end{equation}
with $C$ as in \eqref{4.1}.

If $\sigma_1,\sigma_2\in \dot{\Sigma}$, then similarly to \eqref{4.1} and \eqref{4.2} we have
\begin{align*}
&\|\sigma_1^{\pm 1/2}-\sigma_2^{\pm 1/2}\|_{\mathscr{B}(H^{1/2}(\Gamma))}\le C\|\sigma_1^{\pm 1/2}-\sigma_2^{\pm 1/2}\|_{C^1(\Gamma)},
\\
& \|(\sigma_1^{\pm 1/2}-\sigma_2^{\pm 1/2})^\ast\|_{\mathscr{B}(H^{-1/2}(\Gamma))}\le C\|\sigma_1^{\pm 1/2}-\sigma_2^{\pm 1/2}\|_{C^1(\Gamma)},
\end{align*}
where $C=C(n,\Omega)>0$ is a constant. Whence
\begin{align*}
&\|\sigma_1^{\pm 1/2}-\sigma_2^{\pm 1/2}\|_{\mathscr{B}(H^{1/2}(\Gamma))}\le C\|\sigma_1-\sigma_2\|_{C^1(\Gamma)},
\\
& \|(\sigma_1^{\pm 1/2}-\sigma_2^{\pm 1/2})^\ast\|_{\mathscr{B}(H^{-1/2}(\Gamma))}\le C\|\sigma_1-\sigma_2\|_{C^1(\Gamma)},
\end{align*}
where $C=C(n,\Omega,\dot{\kappa})>0$ is a constant. 

These inequalities together with the interpolation inequality in \cite[Lemma 3.2 in page 264]{Al90} (in which we substitute $\tilde{\nu}$ by $\nu$ and we take $\alpha=1-n/p$) imply
\begin{align}
&C\|\sigma_1^{\pm 1/2}-\sigma_2^{\pm 1/2}\|_{\mathscr{B}(H^{1/2}(\Gamma))}\label{4.3}
\\
&\hskip 3cm \le \|\sigma_1-\sigma_2\|_{C(\Gamma)}^{(p-n)/(2p-n)}+\| \partial_\nu(\sigma_1-\sigma_2)\|_{C(\Gamma)},\nonumber
\\
&C \|(\sigma_1^{\pm 1/2}-\sigma_2^{\pm 1/2})^\ast\|_{\mathscr{B}(H^{-1/2}(\Gamma))}\label{4.4}
\\
&\hskip 3cm\le \|\sigma_1-\sigma_2\|_{C(\Gamma)}^{(p-n)/(2p-n)}+\| \partial_\nu(\sigma_1-\sigma_2)\|_{C(\Gamma)},\nonumber
\end{align}
where $C=C(n,\Omega,p,\dot{\kappa})>0$ is a constant.

We have
\[
\left|\int_\Gamma \sigma^{-1}\partial_\nu \sigma gdS(x)\right|\le \|\sigma^{-1}\partial_\nu\sigma \|_{L^2(\Gamma)}\|g\|_{L^2(\Gamma)}\le C \|\sigma^{-1}\partial_\nu\sigma \|_{L^2(\Gamma)}\|g\|_{H^{1/2}(\Gamma)},
\]
with a constant $C=C(n,\Omega)>0$. Hence the multiplication by $\sigma^{-1}\partial_\nu \sigma$ defines an operator, denoted again by $\sigma^{-1}\partial_\nu \sigma$, acting continuously between $H^{1/2}(\Gamma)$ and $H^{-1/2}(\Gamma)$ and
\[
\|\sigma^{-1}\partial_\nu \sigma \|_{\mathscr{B}(H^{1/2}(\Gamma),H^{-1/2}(\Gamma))}\le C\|\sigma^{-1}\partial_\nu\sigma \|_{L^2(\Gamma)},
\]
where $C$ is as in the  inequality above.

We have similarly
\begin{equation}\label{4.5}
\|\sigma_1^{-1}\partial_\nu \sigma_1-\sigma_2^{-1}\partial_\nu \sigma_2\|_{\mathscr{B}(H^{1/2}(\Gamma),H^{-1/2}(\Gamma))}\le C\|\sigma_1^{-1}\partial_\nu \sigma_1-\sigma_2^{-1}\partial_\nu \sigma_2 \|_{L^2(\Gamma)}.
\end{equation}
Here the constant $C$ is the same constant as in the preceding inequality.

We associate to $\sigma\in \dot{\Sigma}$ the function $q_\sigma =\sigma^{-1/2}\Delta \sigma^{1/2}\in L^p(\Omega)$. The usual Liouville's transform shows that $v_\sigma (g)=\sigma^{1/2}u_\sigma (\sigma^{-1/2}g)$, $g\in H^{1/2}(\Gamma)$, is the unique  solution of the BVP
\[
\left\{
\begin{array}{ll}
-\Delta v+q_\sigma v=0\quad \mbox{in}\; \Omega ,
\\
v_{|\Gamma}=g.
\end{array}
\right.
\]

It is worth noticing that the preceding transform guarantees that  $0$ is not an eigenvalue of the operator $A_\sigma= -\Delta +q_\sigma$, with domain $D(A_\sigma )=D=\{w\in H_0^1(\Omega);\; \Delta u\in L^2(\Omega)\}$, that we consider as an unbounded operator on $L^2(\Omega)$. In this definition we used the fact that $H_0^1(\Omega)$ is continuously embedded in $L^{2n/(n-2)}(\Omega)$, which combined with H\"older's inequality, yields 
\[
\|q_\sigma w\|_{L^2(\Omega)}\le \|q_\sigma\|_{L^n(\Omega)}\|w\|_{L^{2n/(n-2)}(\Omega)}\le C\|q_\sigma\|_{L^p(\Omega)}\|w\|_{H_0^1(\Omega)},
\]
for some constant $C=C(n,\Omega,p)>0$.

We also recall that trace operator $w\in D\rightarrow \partial_\nu w\in H^{-1/2}(\Gamma)$ defines a bounded operator, with
\[
\|\partial_\nu w\|_{H^{-1/2}(\Gamma)}\le c_\Omega\left(\|w\|_{H_0^1(\Omega)}+\|\Delta u\|_{L^2(\Omega)}\right),\quad w\in D
\]
(e.g. \cite[Lemma 2.2, page 131]{Ka}).

Let $\dot{\Lambda}_\sigma\in \mathscr{B}(H^{1/2}(\Gamma ),H^{-1/2}(\Gamma))$ be the operator acting as follows
\[
\dot{\Lambda}_\sigma (g)=\partial_\nu v_\sigma (g),\quad g\in H^{1/2}(\Gamma).
\]
We have the following known formula, that one can also establish in a straightforward manner,
\begin{equation}\label{4.6}
\dot{\Lambda}_\sigma = (\sigma^{-1/2})^\ast \circ \Lambda_\sigma \circ \sigma ^{1/2}+\sigma^{-1}\partial_\nu \sigma /2.
\end{equation}

We get  by putting together inequalities \eqref{4.3} to \eqref{4.6}
\[
C\|\dot{\Lambda}_{\sigma_1}-\dot{\Lambda}_{\sigma_2}\|\le \|\Lambda_{\sigma_1}-\Lambda_{\sigma_2}\|+\|\sigma_1-\sigma_2\|_{C(\Gamma)}^{(p-n)/(2p-n)}+\| \partial_\nu(\sigma_1-\sigma_2)\|_{C(\Gamma)}
\]
which, in light of Theorem \ref{theorem-i1}, gives
\begin{equation}\label{4.7}
C\|\dot{\Lambda}_{\sigma_1}-\dot{\Lambda}_{\sigma_2}\|\le \|\Lambda_{\sigma_1}-\Lambda_{\sigma_2}\|^{(p-n)/(2p-n)}.
\end{equation}

Noting that
\[
\|q_{\sigma_1}-q_{\sigma_2}\|_{H^{-1}(\Omega)}\le \|(q_{\sigma_1}-q_{\sigma_2})\chi_\Omega\|_{H^{-1}(\mathbb{R}^n)},
\]
we modify slightly the proof \cite[Theorem 3.2 in 14]{Ch19} in order to obtain
\[
C\|q_{\sigma_1}-q_{\sigma_2}\|_{H^{-1}(\Omega)}\le \rho^{-2/(2+n)}+e^{c\rho}\|\dot{\Lambda}_{\sigma_1}-\dot{\Lambda}_{\sigma_2}\|,\quad \rho \ge \rho_0,
\]
where $C=C(n,\Omega ,\kappa ,p)>0$, $c=c(n,\Omega ,\kappa ,p)>0$ and $\rho_0=\rho_0(n,\Omega ,\kappa ,p)>0$ are constants.
\\
This and \eqref{4.7} give
\begin{equation}\label{4.8}
C\|q_{\sigma_1}-q_{\sigma_2}\|_{H^{-1}(\Omega)}\le \rho^{-2/(2+n)}+e^{c\rho}\|\Lambda_{\sigma_1}-\Lambda_{\sigma_2}\|^{(p-n)/(2p-n)},\quad \rho \ge \rho_0,
\end{equation}
where the constant $C$, $c$ and $\rho_0$ are the same as above.

\begin{lemma}\label{lemma4.1}
Let $a\in L^\infty(\Omega)$ satisfying $\lambda^{-1}\le a\le \lambda$, for some constant $\lambda\ge 1$. We have, for any $w\in H^2(\Omega)$, 
\[
C\|w\|_{H^1(\Omega)}\le \|\mathrm{div}(a\nabla w)\|_{H^{-1}(\Omega)}+\|w\|_{L^2(\Gamma)}+\|\nabla w\|_{L^2(\Gamma)},
\]
where $C=C(n,\Omega ,\lambda)>0$ is a constant.
\end{lemma}

\begin{proof}
Let $w\in H^2(\Omega)$ and $\tilde{w}=\mathcal{E}(w_{|\Gamma})\in H^1(\Omega)$. Since 
\[
C\|w\|_{H^{1/2}(\Gamma)}\le \|w\|_{H^1(\Gamma)}\le \|w\|_{L^2(\Gamma)}+\|\nabla w\|_{L^2(\Gamma)}
\]
and 
\[
C \|\mathrm{div}(a\nabla \tilde{w})\|_{H^{-1}(\Omega)}\le\|\tilde{w}\|_{H^1(\Omega)}=\|w\|_{H^{1/2}(\Gamma)},
\]
we derive that
\[
C \|\mathrm{div}(a\nabla w)\|_{H^{-1}(\Omega)}\le \|w\|_{L^2(\Gamma)}+\|\nabla w\|_{L^2(\Gamma)}.
\]

In the sequel, we endow $H_0(\Omega)$ with the norm $\psi \mapsto \|\nabla \psi\|_{L^2(\Omega)}$. As $w-\tilde{w}\in H_0^1(\Omega)$, we get
\[
\int_\Omega a|\nabla (w-\tilde{w})|^2dx =\langle\mathrm{div}(a\nabla w)-\mathrm{div}(a\nabla \tilde{w})|w-\tilde{w}\rangle_1.
\]
Hence
\[
\lambda^{-1}\|w-\tilde{w}\|_{H^1(\Omega)}\le \|\mathrm{div}(a\nabla w)\|_{H^{-1}(\Omega)}+\|\mathrm{div}(a\nabla \tilde{w})\|_{H^{-1}(\Omega)},
\]
from which we deduce in a straightforward manner
\[
C\|w-\tilde{w}\|_{H^1(\Omega)}\le \|\mathrm{div}(a\nabla w)\|_{H^{-1}(\Omega)}+\|w\|_{L^2(\Gamma)}+\|\nabla w\|_{L^2(\Gamma)}.
\]
where $C=C(n,\Omega ,\lambda)>0$ is a  constant.

The last inequality, together with the following one
\[
\|w\|_{H^1(\Omega)}\le \|\tilde{w}\|_{H^1(\Omega)}+\|w-\tilde{w}\|_{H^1(\Omega)},
\]
then give the expected inequality.
\end{proof}

\begin{proposition}\label{proposition4.1}
For each $\sigma_1,\sigma_2\in \dot{\Sigma}$, we have
\begin{equation}\label{4.9}
C\|\sigma_1-\sigma_2\|_{H^1(\Omega )}\le \|q_{\sigma_1}-q_{\sigma_2}\|_{H^{-1}(\Omega)}+\|\sigma_1-\sigma_2\|_{L^2(\Gamma)}+\|\nabla (\sigma_1-\sigma_2)\|_{L^2(\Gamma)},
\end{equation}
where $C=C(n,\Omega ,p,\dot{\kappa})>0$ is a constant.
\end{proposition}

\begin{proof}
In this proof $C=C(n,\Omega ,p,\dot{\kappa})>0$ denotes a generic constant. Let $\sigma_1,\sigma_2\in \dot{\Sigma}$ and set $a=\sqrt{\sigma_1\sigma_2}$, $f=2a(q_{\sigma_1}-q_{\sigma_2})$, and $w=\ln (\sigma_1/\sigma_2)$. From the calculations in \cite{Al88} or \cite{SU86}, we derive 
\[
\mathrm{div}(a\nabla w)=f.
\]
However, the calculations that lead to this equation can be carried out easily.

We apply Lemma \ref{lemma4.1} in order to get
\[
C\|w\|_{H^1(\Omega)}\le \|q_{\sigma_1}-q_{\sigma_2}\|_{H^{-1}(\Omega)}+\|w\|_{L^2(\Gamma)}+\|\nabla w\|_{L^2(\Gamma)}.
\]
The following identities
\begin{align*}
&w(x)=(\sigma_1(x)-\sigma_2(x))\int_0^1\frac{dt}{\sigma_2(x)+t(\sigma_1(x)-\sigma_2(x))},\quad x\in \overline{\Omega},
\\
&\nabla (\sigma_1(x)-\sigma_2(x))=\sigma_1(x)[\nabla w+(1/\sigma_1(x)-1/\sigma_2(x))\nabla \sigma_2(x)] ,\quad x\in \overline{\Omega},
\end{align*}
yield easily the expected inequality.
\end{proof}

We end up getting by putting together \eqref{thm-i1.1}, \eqref{thm-i1.2}, \eqref{4.8} and \eqref{4.9} 
\begin{equation}\label{4.10}
C\|\sigma_1-\sigma_2\|_{H^1(\Omega )}\le \rho^{-2/(2+n)}+e^{c\rho}\|\Lambda_{\sigma_1}-\Lambda_{\sigma_2}\|^{(p-n)/(2p-n)},\quad \rho \ge \rho_0,
\end{equation}
where $C$, $c_0$ and $\rho_0$ are the same as in \eqref{4.8}.

The proof of Theorem \ref{theorem-i2} follows by using a usual minimizing argument, in the right hand side of \eqref{4.10}, with respect to $\rho$.

\subsection{Proof of Theorems \ref{theorem-i3}}\label{subsection3.2}

We first proceed to the construction of $\mathscr{D}$. 

The unit cube $]0,1[^n$ is denoted by $\mathscr{Q}_0$. Recall that the Bernstein's polynomials $p_{k,j}$ are given by
\[
p_{k,j}(t)=C_k^jt^j(1-t)^{k-j},\quad 0\le j\le k,
\]
with
\[
C_k^j=\frac{k!}{j!(k-j)!}.
\]
To $f\in C(\overline{\mathscr{Q}_0})$, we associate the  Bernstein polynomial
\[
B_k^0(f)(t_1,\ldots ,t_n)=\sum_{j_1=0}^k\ldots \sum_{j_n=0}^k f\left( \frac{j_1}{k},\ldots ,\frac{j_n}{k}\right)p_{k,j_1}(t_1)\ldots p_{k,j_n}(t_n).
\]
\begin{theorem}\label{theorem2}
$($\cite[Theorem 1.2.9, page 18]{BP}$)$ For any $f\in C(\overline{\mathscr{Q}_0})$, we have
\[
\lim_{k\rightarrow \infty}\|f-B_k^0(f)\|_{C(\overline{\mathscr{Q}_0})}=0.
\]
\end{theorem}

Fix $a<b$ and denote by $\mathscr{Q}$ the cube $]a,b[^n$. We associate to each $f\in C(\overline{\mathscr{Q}})$ the polynomial
\begin{align*}
&B_k(f)(x_1,\ldots ,x_n)=\sum_{j_1=0}^k\ldots \sum_{j_n=0}^k f\left( a+\frac{j_1}{k}(b-a),\ldots ,a+\frac{j_n}{k}(b-a)\right)
\\
&\hskip 6.5cm \times p_{k,j_1}\left(\frac{x_1-a}{b-a}\right)\ldots p_{k,j_n}\left(\frac{x_n-a}{b-a}\right).
\end{align*}

The following result is a straightforward consequence of Theorem \ref{theorem2}.

\begin{corollary}\label{corollary1}
 For any $f\in C(\overline{\mathscr{Q}})$, we have
\[
\lim_{k\rightarrow \infty}\|f-B_k(f)\|_{C(\overline{\mathscr{Q}})}=0.
\]
\end{corollary}

If $O$ is an open bounded subset of $\mathbb{R}^n$, we set
\[
C_+(\overline{O})=\{ \sigma \in C(\overline{O});\; \sigma >0\; \mbox{in}\; \overline{O}\}.
\]

Let the cube $\mathscr{Q}$  be chosen so that $\Omega \Subset \mathscr{Q}$. Then according to Tietze extension theorem (e.g. \cite[Theorem 9.35, page 256]{BrP}) for each $\sigma \in C_+(\overline{\Omega})$ there exists $\sigma_e\in C_+(\overline{\mathscr{Q}})$ so that $\sigma_e=\sigma$ in $\overline{\Omega}$ and 
\[
\|\sigma_e\|_{C(\overline{\mathscr{Q}})}=\|\sigma\|_{C(\overline{\Omega})}.
\]

In the sequel we shall use the fact that $B_k(\sigma_e)\in C_+(\overline{\mathscr{Q}})$, whenever $\sigma \in C_+(\overline{\Omega})$. Define
 \[
 \mathscr{D}=\{\chi=B_k(\sigma_e)_{|\overline{\Omega}};\;  k\in \mathbb{N},\; \sigma \in C_+(\overline{\Omega})\}.
 \]
 Pick $\epsilon >0$ and $\sigma \in C_+(\overline{\Omega})$. In light of Corollary \ref{corollary1}, we can choose $k\in \mathbb{N}$ sufficiently large so that
 \[
 \|\sigma_e-B_k(\sigma_e)\|_{C(\overline{\mathscr{Q}})}\le \epsilon .
 \]
In consequence, we have, where $\chi =B_k(\sigma_e)_{|\overline{\Omega}}\in \mathscr{D}$,
\[
\|\sigma-\chi\|_{C(\overline{\Omega})}\le \|\sigma_e-B_k(\sigma_e)\|_{C(\overline{\mathscr{Q}})}\le \epsilon .
\]
In other words, we proved that $\mathscr{D}$ is dense in $C_+(\overline{\Omega} )$ with respect to the topology of $C(\overline{\Omega})$.

We now complete the proof of Theorem \ref{theorem-i3}. Let $\chi_j\in \mathscr{D}$, $j=1,2$, so that $\Lambda_{\chi_1}=\Lambda_{\chi_2}$. Then it is straightforward to check that $\chi_1,\chi_2$ belong to $\Sigma$, for some $\kappa =\kappa (\chi_1,\chi_2)>1$. We end up getting $\chi_1=\chi_2$ by applying Theorem \ref{theorem-i2}.

\begin{remark}
{\rm
There is another possibility to construct $\mathscr{D}$ by using a sequence of mollifiers and the convolution. Let $\psi\in C_0^\infty(\mathbb{R}^n)$ satisfying $0\le \psi $, $\mbox{supp}(\psi)\subset B(0,1)$ and $\int_{\mathbb{R}^n}\psi (x)dx=1$. For each integer $k\ge 1$, we define $\psi_k$ by $\psi_k(x)=k^n\psi (kx)$, $x\in \mathbb{R}^n$. If $f\in C(\overline{\mathscr{Q}})$ then $f_k=\psi_k\ast f$ is well defined on $\mathscr{Q}_k=\{x\in \mathscr{Q};\; \mbox{dist}(x,\partial \mathscr{Q})>1/k\}$. We derive from \cite[Theorem 1.6, page 5]{Ch13} that $\|f_k-f\|_{C(\overline{\Omega})}$ converge to zero as $k$ goes to $\infty$. We can therefore proceed as above to prove that
\[
\mathscr{D}=\{\chi=(\psi_k\ast \sigma_e)_{|\overline{\Omega}};\;  k\in \mathbb{N},\; \sigma \in C_+(\overline{\Omega})\}
\]
is dense in $C_+(\overline{\Omega})$, when this later is equipped with the norm of $C(\overline{\Omega})$.
}
\end{remark}

\section{Additional results}\label{section4}

\subsection{Anisotropic case: determination of the conformal factor}\label{subsection4.1}

We describe the main ideas to extend some results of the isotropic case to that of the anisotropic case. We are mainly concerned with the determination of the conformal factor. To this end we fix $A=(a^{ij})$ a matrix valued function whose coefficients belong to $C^{1,\alpha}(\overline{\Omega})$. We suppose that $A$ is symmetric and satisfies, for some $\mu > 1$,
\[
\mu^{-1}|\xi|^2\le A(x)\xi \cdot \xi\le \mu|\xi|^2, \quad x\in \Omega, \;  \xi \in \mathbb{R}^n,
\]
and
\[
\max_{1\le i,j\le n}\|a^{ij}\|_{C^{1,\alpha}(\overline{\Omega})}\le \mu.
\]

Consider the BVP
\begin{equation}\label{6.1}
\left\{
\begin{array}{l}
\mbox{div}(\sigma A\nabla u)=0\quad \mbox{in}\; \Omega,
\\
u_{|\Gamma}=g.
\end{array}
\right.
\end{equation}
We can proceed similarly to the isotropic case to show that, for any $\sigma \in \Sigma$ and $g\in H^{1/2}(\Gamma)$, the BVP \eqref{6.1} possesses a unique solution
$\tilde{u}_\sigma (g)\in H^1(\Omega )$. Furthermore, we can define the Dirichlet-to-Neumann map, associated to $\sigma$, as the bounded operator given by
\begin{align*}
\tilde{\Lambda}_\sigma :g\in H^{1/2}(\Gamma )&\rightarrow H^{-1/2}(\Gamma) : 
\\
&\langle\Lambda_\sigma (g),h\rangle_{1/2}=\int_\Omega \sigma A\nabla \tilde{u}_\sigma(g)\cdot \nabla\mathcal{E}hdx,\quad h\in H^{1/2}(\Gamma),
\end{align*}
which satisfies
\[
\|\tilde{\Lambda} _\sigma\|\le C,
\]
for some constant $C=C(n,\Omega ,\kappa ,\mu)$.

The canonical parametrix associated to the operator $\mbox{div}(\sigma A\nabla \cdot)$, with $\sigma\in \Sigma$, is given by
\[
H_\sigma (x,y)=\frac{[A^{-1}(y)(x-y)\cdot (x-y)]^{(2-n)/2}}{(n-2)|\mathbb{S}^{n-1}|\sigma(y)[\mbox{det}A(y)]^{1/2}},\quad x,y\in \mathbb{R}^n,\; x\ne y.
\]
Here $\sigma$ and $A$ are extended according to Theorem \ref{exth} (\cite[Formula (2.4) in page 258]{Kal}). Elementary computations show that, for all $\sigma_1,\sigma_2\in \Sigma$ and $x,y\in \mathbb{R}^n$ with $ x\ne y$, we have
\begin{equation}\label{6.2}
\mathfrak{c}^{-1} |x-y|^{2-2n}\le A(x)\nabla_xH_{\sigma_1}(x,y)\cdot\nabla_xH_{\sigma_2}(x,y)\le \mathfrak{c}|x-y|^{2-2n},
\end{equation}
where $\mathfrak{c}=\mathfrak{c}(n,\Omega ,\kappa,\mu )>1$ is a constant. Set
\[
\tilde{\mathscr{S}}_\sigma=\{ u\in C^2(\overline{\Omega});\; \mbox{div}(\sigma A\nabla u)=0\},\quad \sigma \in \Sigma.
\]
As for Theorem  \ref{theorem1},  we have as a consequence of \cite[Theorem 5, page 282]{Kal} the following result.

\begin{theorem}\label{theorem6.1}
For any $\sigma\in \Sigma_\kappa$ and $y\in \Omega_0\setminus \overline{\Omega}$, there exists $\tilde{u}_{\sigma} ^y\in \tilde{\mathscr{S}}_\sigma$ so that
\begin{align*}
&|\tilde{u}_{\sigma} ^y(x)-H_\sigma(x ,y)|\le C|x-y|^{2-n+\alpha},\quad x\in \overline{\Omega},
\\
&|\nabla \tilde{u}_{\sigma} ^y(x)-\nabla H_\sigma (x,y)|\le C|x-y|^{1-n+\alpha},\quad x\in \overline{\Omega},
\end{align*}
where  $C=C(n,\Omega ,\kappa,\mu,\alpha)>0$ is a constant.
\end{theorem}

On the other hand, as for the isotropic case, for all $\sigma_j\in \Sigma$ and $u_j\in \tilde{\mathscr{S}}_{\sigma_j}$, $j=1,2$, the following identity holds
\begin{equation}\label{6.3}
\int_\Omega (\sigma_1-\sigma_2) A\nabla u_1\cdot \nabla u_2dx=\langle (\tilde{\Lambda}_{\sigma_1}-\tilde{\Lambda}_{\sigma_2})v_1,v_2\rangle_{1/2},
\end{equation}
where we set $v_j=u_j-\int_\Omega u_jdx$, $j=1,2$.

In light of \eqref{6.2}, \eqref{6.3} and Theorem \ref{theorem6.1} we can mimic the proof of Theorem \ref{theorem-i1} and Corollary \ref{corollary-i1} in order to obtain the following theorem (we observe that Theorem \ref{thm-Al90} still holds if $L_{\sigma_j}$ is substituted by the operator $\mathrm{div}(\sigma_jA\cdot \nabla \cdot)$).

\begin{theorem}\label{theorem-6.2}
If $\Omega$ is of class $C^{1,1}$ then, for all $\sigma_1,\sigma_2\in \Sigma$, we have
\begin{align*}
& \|\sigma_1-\sigma_2 \|_{C(\Gamma)}\le C\|\tilde{\Lambda}_{\sigma_1}- \tilde{\Lambda}_{\sigma_2}\|,
\\
& \|\nabla (\sigma_1-\sigma_2) \|_{C(\Gamma)}\le C\|\tilde{\Lambda}_{\sigma_1}- \tilde{\Lambda}_{\sigma_2}\|^{\alpha/(\alpha+1)}, 
\end{align*}
where $C=C(n,\Omega ,\kappa, \mu,\alpha )>0$ is a constant.
\end{theorem}

\begin{lemma}\label{lemma6.1}
Let $\ell \ge 2$ be an integer and $f\in C^{\ell,\alpha}(\overline{\Omega_{\varrho_0}})$, for some $\varrho_0>0$, satisfying, for some $\varkappa >0$,
\[
\|f\|_{C^{\ell,\alpha}(\overline{\Omega_{\varrho_0}})}\le \varkappa.
\]
Let $x_0\in \Gamma$ so that
\[
(-1)^\ell\partial_\nu ^\ell f(x_0)=|\partial_\nu ^\ell f(x_0)|.
\]
Then the following inequality holds
\begin{align*}
|\partial_\nu ^\ell f(x_0)|\mathrm{dist}(x,\Gamma)^\ell \le f(x)&-f(\mathfrak{p}(x))
\\
&+\sum_{j=1}^{\ell-1}(-1)^{j+1}\partial_\nu^jf(\mathfrak{p}(x))\mathrm{dist}(x,\Gamma)^j+\varkappa '|x-x_0|^{\ell+\alpha},
\end{align*}
where $\varkappa'=\varkappa'(n,\Omega ,\varkappa,\ell)>0$ is a constant.
\end{lemma}

\begin{proof}
We use Taylor's formula and we proceed as in the proof of Lemma \ref{lemma2}.
\end{proof}

We set, for fixed $\varrho_0>0$, $\Sigma^0=\Sigma^1=\Sigma$ and
\[
\Sigma^\ell=\{ \sigma \in \Sigma \cap C^{\ell,\alpha}(\overline{\Omega_{\varrho_0}}),\; \|\sigma\|_{C^{\ell,\alpha}(\overline{\Omega_{\varrho_0}})}\le \kappa\},\quad \ell \ge 2.
\]
We also introduce the notations
\[
\gamma_0=1\quad \mbox{and}\quad \gamma_j=\prod_{j=1}^\ell \frac{\alpha}{\alpha +j},\; \ell \ge 1.
\]
An extension of the proof \eqref{thm-i1.2} of Theorem \ref{theorem-i1} together with an induction argument with respect to $\ell$ yield the following result.

\begin{theorem}\label{theorem6.3}
Suppose that $\Omega$ is of class $C^{1,1}$ and $\ell$ is a non negative integer. We have, for all $\sigma_1,\sigma_2\in \Sigma^\ell$,
\[
 \|\partial_\nu ^\ell (\sigma_1-\sigma_2) \|_{C(\Gamma)}\le C\|\tilde{\Lambda}_{\sigma_1}- \tilde{\Lambda}_{\sigma_2}\|^{\gamma_\ell}, 
\]
where $C=C(n,\Omega ,\kappa,\varrho_0, \alpha ,\ell)>0$ is a constant
\end{theorem}

The following lemma is obtained by iterating \cite[Lemma 3.2 in page 264]{Al90}.

\begin{lemma}\label{lemma6.2}
Let $\ell \ge 2$ an integer and $f\in C^{\ell,\alpha}(\overline{\Omega_{\varrho_0}})$. Then
\[
C\max_{|\beta|=\ell} \|\partial^\beta f\|_{C(\Gamma)}\le \|\partial_\nu ^{\ell-1}f\|_\ast^{\gamma_1}+\|\partial_\nu ^{\ell-2}f\|_\ast^{\gamma_1^2}+\ldots +\|f\|_\ast^{\gamma_1^\ell},
\]
where $C=C(n,\Omega ,\varrho_0,\ell )>0$ is a constant and
\[
\|f\|_\ast =\|\partial_\nu f\|_{C(\Gamma)}+\|f\|_{C(\Gamma)}.
\]
\end{lemma}

In light of this lemma, Theorem \ref{theorem6.3} imply in a straightforward manner the following corollary.

\begin{corollary}\label{corollary6.3}
Assume that $\Omega$ is of class $C^{1,1}$ and let $\ell\ge 2$ be an integer. We have, for all $\sigma_1,\sigma_2\in \Sigma^\ell$,
\[
 \max_{|\beta|=\ell}\|\partial ^\beta (\sigma_1-\sigma_2) \|_{C(\Gamma)}\le C\|\tilde{\Lambda}_{\sigma_1}- \tilde{\Lambda}_{\sigma_2}\|^{\gamma_1\gamma_\ell}, 
\]
where $C=C(n,\Omega ,\kappa,\varrho_0, \alpha ,\ell)>0$ is a constant
\end{corollary}

We mention that the case of general anisotropic conductivities can be reformulated as a geometric inverse problem. Precisely the problem is to know whether it is possible to recover the metric of a compact Riemannian manifold with boundary, from the corresponding Dirichlet-to-Neumann map. This problem was solved by Guillarmou and Tzou in dimension two \cite{GT}. In dimensions greater or equal to three the answer is positive for conformally transversally anisotropic manifolds, under the assumption  that  the  geodesic  X-ray  transform  on  the  transversal  manifold  is injective \cite{DKLS,DKLLS}. Recent progress toward solving the general case can be found in \cite{KLS}.
 
 Concerning stability inequalities we refer to the earlier work by Kang and Yun \cite{KY} in which the authors provide H\"older stability inequality at the boundary of anisotropic conductivities from local Dirichlet-to-Neumann map. While Caro and Salo \cite{CS} obtained logarithmic type stability inequality for the conformal factor in admissible geometries. 
 
For non uniqueness results on the determination of anisotropic conductivities from partial boundary data we refer to the recent paper by Daud\'e, Kamran and Nicoleau \cite{DKN} and references therein.

\subsection{Isotropic case with partial data}\label{subsection4.2}

Throughout this section we use the same notations as in Sections \ref{section2} and \ref{section3}. Fix $\hat{x}\in \mathbb{R}^n$ outside the closure of the convex hull of $\Omega$ and denote by $\Gamma_0$ an open neighborhood of the set
\[
F=\{ x\in \Gamma ;\; (x-\hat{x})\cdot \nu (x)\le 0\}.
\]
Pick $\chi \in C_0^\infty (\Gamma_0)$ so that $0\le \chi \le 1$ and $\chi=1$ in a neighborhood of $F$. We then introduce the following partial Dirichlet-to-Neumann map
\[
\hat{\Lambda}_\sigma =\chi \Lambda_\sigma ,\quad \sigma \in \Sigma .
\]
 
We consider the following subset of $\dot{\Sigma}$, where $0<s<1/2$,
\[
\hat{\Sigma}=\left\{ \sigma \in W^{2,\infty}(\mathbb{R}^n)\cap H^{2+s}(\mathbb{R}^n);\; \mbox{supp}(\sigma)\subset \overline{\Omega},
\; \|\sigma \|_{W^{2,\infty}(\mathbb{R}^n)\cap H^{2+s}(\mathbb{R}^n)}\le \kappa\right\}.
\]
In the sequel we use  that $\hat{\Sigma}$ is continuously embedded in $C^{1,1/2}(\overline{\Omega})$.

Let $\sigma_1,\sigma_2\in \hat{\Sigma}$, $a=\sqrt{\sigma_1\sigma_2}$, $f=2a(q_{\sigma_1}-q_{\sigma_2})$ and $w=\ln (\sigma_1/\sigma_2)$. As we have seen in the proof of Proposition \ref{proposition4.1}, $w$ is the solution of the BVP
\[
\mathrm{div}(a\nabla w)=f.
\]
From the results in \cite[Section 4.5 in page 168]{Chou}, there exists three constants $C=C(n,\Omega ,\kappa ,\Gamma_0,s)$, $c=c(n,\Omega ,\kappa ,\Gamma_0,s)$ and $\beta=\beta(n,\Omega ,\kappa ,\Gamma_0,s)$ so that, for any $0<\epsilon<1$, we have
\begin{align*}
&C\left(\|w\|_{C(\Gamma)}+\|\nabla w\|_{C(\Gamma)}\right)\le \epsilon ^\beta\|w\|_{C^{1,\alpha}(\overline{\Omega})}
\\
&\hskip 3cm + e^{c/\epsilon}\left( \|w\|_{L^2(\Gamma)}+\|\nabla w\|_{L^2(\Gamma_0)}+\|q_{\sigma_1}-q_{\sigma_2}\|_{L^2(\Omega)}\right),
\end{align*}
from which we derive similarly as in the proof of Proposition \ref{proposition4.1}
\begin{align*}
&C\left(\|\sigma_1-\sigma_2\|_{C(\Gamma)}+\|\nabla (\sigma_1-\sigma_2)\|_{C(\Gamma)}\right)\le \epsilon ^\beta 
\\
&\hskip 1cm + e^{c/\epsilon}\left( \|\sigma_1-\sigma_2\|_{L^2(\Gamma_0)}+\|\nabla (\sigma_1-\sigma_2)\|_{L^2(\Gamma_0)}+\|q_{\sigma_1}-q_{\sigma_2}\|_{L^2(\Omega)}\right).
\end{align*}
One more time, Proposition \ref{proposition4.1} yields
\begin{align}
&C\|\sigma_1-\sigma_2\|_{H^1(\Omega)}\le \epsilon ^\beta \label{7.1}
\\
&\hskip 1cm + e^{c/\epsilon}\left( \|\sigma_1-\sigma_2\|_{L^2(\Gamma_0)}+\|\nabla (\sigma_1-\sigma_2)\|_{L^2(\Gamma_0)}+\|q_{\sigma_1}-q_{\sigma_2}\|_{L^2(\Omega)}\right).\nonumber
\end{align}

Denote by $\ell \ge 2 $ the smallest integer satisfying 
\[
\frac{\ell+1}{\ell-1}\le 1+\alpha.
\]
According to \cite[Theorem 1.4 in page 724]{KY}, we get
\begin{equation}\label{7.2}
\|\sigma_1-\sigma_2\|_{C(\Gamma_0)}+\|\nabla (\sigma_1-\sigma_2)\|_{C(\Gamma_0)}\le C\|\hat{\Lambda}_{\sigma_1}-\hat{\Lambda}_{\sigma_2}\|^{2^{-\ell}}.
\end{equation}
Here and henceforward $C=C(n,\Omega ,\kappa ,\Gamma_0,s)>0$ is a constant.

On the other hand, we have from \cite[Theorem 1.1 in page 2461]{CDR}
\begin{equation}\label{7.3}
\|q_{\sigma_1}-q_{\sigma_2}\|_{L^2(\Omega)}\le C\left|\ln \left|\ln \|\hat{\Lambda}_{\sigma_1}-\hat{\Lambda}_{\sigma_2}\|\right|\right|^{-2s/(3+3s)},
\end{equation}
if $\|\hat{\Lambda}_{\sigma_1}-\hat{\Lambda}_{\sigma_2}\|\le \Lambda_0$, for some constant $0<\Lambda_0=\Lambda_0(n,\Omega ,\kappa ,\Gamma_0,s)<e^{-1}$.

Inequalities \eqref{7.2} and \eqref{7.3} in \eqref{7.1} give, where $0<\epsilon <1$,
\begin{equation}\label{7.4}
C\|\sigma_1-\sigma_2\|_{H^1(\Omega)}\le \epsilon ^\beta+e^{c/\epsilon}\left|\ln \left|\ln \|\hat{\Lambda}_{\sigma_1}-\hat{\Lambda}_{\sigma_2}\|\right|\right|^{-2s/(3+3s)},
\end{equation}
whenever $\|\hat{\Lambda}_{\sigma_1}-\hat{\Lambda}_{\sigma_2}\|\le \Lambda_0$.

Define $\Psi_{c,\beta} :[0,\infty )\rightarrow [0,\infty)$, where $0<c<e^{-e}$ and $\beta >0$, as follows
\[
\Psi_\beta (rho)=\left\{
\begin{array}{ll}
0 &\mbox{if}\rho =0,
\\
|\ln |\ln |\ln \rho|||^{-\beta}\quad &\mbox{if}\; \rho \in (0,c),
\\
\rho &\mbox{if}\; \rho \in [c,\infty ).
\end{array}
\right.
\]
We obtain, by minimizing the right hand side of \eqref{7.4} with respect to $\epsilon$, the following result.

\begin{theorem}\label{theorem7.1}
Suppose that $\Omega$ is of class $C^{1,1}$. Then there exist three constants $C=C(n,\Omega ,\kappa ,\Gamma_0,s)>0$, $0<c=c(n,\Omega ,\kappa ,\Gamma_0,s)<e^{-e}$ and $\beta=\beta(n,\Omega ,\kappa ,\Gamma_0,s)>0$ so that, for any $\sigma_1,\sigma_2 \in \hat{\Sigma}$, we have 
\[
\|\sigma_1-\sigma_2\|_{H^1(\Omega)}\le C\Psi_{c,\beta} \left(\|\hat{\Lambda}_{\sigma_1}-\hat{\Lambda}_{\sigma_2}\|\right).
\]
\end{theorem}

\section{Stability at the boundary using oscillating solutions}\label{section5}

The following lemma is essentially due to Kohn and Vogelius \cite{KV84}. The version stated here is borrowed from \cite{Ka} (see Lemma 4.1 in page 142 and its proof). We suppose in this section that  $\Omega$ is again of class $C^{1,1}$.

\begin{lemma}\label{lemma-os.1}
Pick $x_0\in \Gamma$. Then there exists a sequence $(\psi_k)$ in $H^{3/2}(\Gamma)\cap C^{1,1}(\Gamma)$ satisfying, for each $k\ge 1$, the following properties:
\\
$\mathrm{(i)}\; \mathrm{supp}(\psi_k)\subset B(x_0,\mathfrak{c}k^{-1})$,  
\\
$\mathrm{(ii)}\;\|\psi_k\|_{H^{1/2}(\Gamma)}=1$,
\\
$\mathrm{(iii)}\;C^{-1}k^{-(1/2+s)}\le \|\psi_k\|_{H^{-s}(\Gamma)}\le Ck^{-(1/2+s)}$, $-1\le s\le 1$,
\\
where $\mathfrak{c}=\mathfrak{c}(n,\Omega)$ and $C=C(n,\Omega ,s)\ge 1$ are constants.
\end{lemma}

For notational convenience, for $\sigma \in \Sigma$, we set, where $k\ge 1$, $u_\sigma ^k=u_\sigma (\psi_k)$ ($\in H^2(\Omega)$), with $\psi_k$ as in Lemma \ref{lemma-os.1}.

\begin{proposition}\label{proposition-os.1}
Let $0<\rho <\mathrm{diam}(\Omega)$. There exist a  constant $C=C(n,\Omega ,\kappa)>0$  so that, for any $x_0\in \Gamma$, we have
\begin{equation}\label{prop-os.1}
\|u_\sigma^k\|_{H^1(\Omega\setminus \overline{B}_\rho)}\le C\rho^{-1}k^{-1},\quad k\ge 2\mathfrak{c}/\rho,
\end{equation}
where $B_\rho=B(x_0,\rho)$ and $\mathfrak{c}$ is as in Lemma \ref{lemma-os.1}.
\end{proposition}

\begin{proof}
Pick $0<\rho <\mathrm{diam}(\Omega)$ and $x_0\in \Gamma$. Fix $\phi \in C_0^\infty (\Omega)$ so that $0\le \phi \le 1$, $\phi =1$ in a neighborhood of $B_{\rho/2}$ and
$|\nabla \phi|\le c\rho^{-1}$, where $c$ is a universal constant. Let then $\varphi=1-\phi$ and $v^k=\varphi u_\sigma^k$. Furthermore, according to Lemma \ref{lemma-os.1}, we have $\mathrm{supp}(\psi_k)\subset B_{\rho/2}$, for each $k\ge k_\rho=2\mathfrak{c}/\rho$. In consequence, $v_k\in H_0^1(\Omega)$, for each $k\ge k_\rho$.

We assume in the rest of this proof that $k\ge k_\rho$. We have
\begin{align*}
\mathrm{div}(\sigma \nabla v^k)&=\mathrm{div}(\sigma \varphi \nabla u_\sigma^k)+\mathrm{div}(\sigma u_\sigma^k\nabla \varphi) 
\\
& =\sigma \nabla \varphi\cdot \nabla u_\sigma^k+\mathrm{div}(\sigma u_\sigma^k\nabla \varphi).
\end{align*}
Green's formula then yields
\begin{align*}
\int_\Omega \sigma |\nabla v^k|^2dx&= -\int_\Omega \sigma \varphi u_\sigma^k \nabla \varphi\cdot \nabla u_\sigma^kdx +\int_\Omega \sigma u_\sigma^k\nabla \varphi \cdot \nabla (\varphi u_\sigma^k)dx
\\
&=\int_\Omega \sigma (u_\sigma^k)^2|\nabla \varphi|^2dx=\int_\Omega \sigma (u_\sigma^k)^2|\nabla \phi|^2dx.
\end{align*}
Whence
\begin{equation}\label{os1}
\int_{\Omega\setminus \overline{B}_\rho} \sigma |\nabla u_\sigma^k|^2dx\le \kappa ^2\int_\Omega u_\sigma^k (u_\sigma^k |\nabla \phi|^2)dx.
\end{equation}
We write 
\[
\int_\Omega u_\sigma^k (u_\sigma^k |\nabla \phi|^2)dx=-\int_\Omega u_\sigma^k \mathrm{div}(\sigma \nabla u) dx,
\]
where $u\in H_0^1(\Omega )\cap H^2(\Omega)$ is the unique solution of the equation
\[
-\mathrm{div}(\sigma \nabla u)= u_\sigma^k |\nabla \phi|^2 \quad \mathrm{in}\quad \Omega .
\]
Noting  that $\mathrm{div}(\sigma \nabla u_\sigma^k)=0$ in $\Omega$, we obtain by applying Green's formula
\[
-\int_\Omega u_\sigma^k \mathrm{div}(\sigma \nabla u) dx=-\int_\Gamma \psi_k\sigma \partial_\nu udS(x),
\]
from which we derive
\begin{equation}\label{os2}
\int_\Omega u_\sigma^k (u_\sigma^k |\nabla \phi|^2)dx\le \kappa \|\partial_\nu u\|_{H^{1/2}(\Gamma)}\|\psi_k\|_{H^{-1/2}(\Gamma)}.
\end{equation}
Now, the usual $H^2$ a priori estimate (e.g. \cite[Theorem 8.53 in page 326]{RR} and its proof) together with the continuity of the trace operator $w\in H^2(\Omega)\mapsto \partial_\nu w\in H^{1/2}(\Gamma)$ give
\[
\|\partial_\nu u\|_{H^{1/2}(\Gamma)} \le C\|u_\sigma^k |\nabla \phi|^2\|_{L^2(\Omega)}.
\] 
Here and until the end of the proof $C=C(n,\Omega,\kappa)>0$ denotes a generic constant. Thus
\begin{equation}\label{os3}
\|\partial_\nu u\|_{H^{1/2}(\Gamma)} \le C\rho^{-1}\|u_\sigma^k |\nabla \phi|\|_{L^2(\Omega)}.
\end{equation}
In light of two-sided inequality of Lemma \ref{lemma-os.1}, we get by combining \eqref{os2} and \eqref{os3}
\[
\|u_\sigma^k |\nabla \phi|\|_{L^2(\Omega)}\le C\rho^{-1}k^{-1}.
\]
This in \eqref{os1} gives
\begin{equation}\label{os4}
\|\nabla u_\sigma^k\|_{L^2(\Omega\setminus \overline{B}_\rho)}\le C\rho^{-1}k^{-1}.
\end{equation}
As $v^k\in H_0^1(\Omega)$, we have, according to Poincarr\'e's inequality,
\[
\int_\Omega |v^k|^2dx\le c_\Omega \int_\Omega |\nabla v^k|^2.
\]
This and the preceding calculations yield
\begin{equation}\label{os5}
\| u_\sigma^k\|_{L^2(\Omega\setminus \overline{B}_\rho)}\le C\rho^{-1}k^{-1},\quad k\ge k_\rho.
\end{equation}
We obtain the expected inequality by putting together \eqref{os4} and \eqref{os5}.
\end{proof}

\begin{lemma}\label{lemma-os.2}
Let $0<\rho <\mathrm{diam}(\Omega)$. There exists a constant $C=C(n,\Omega ,\kappa)>0$ so that, for any $x_0\in \Gamma$, we have
\begin{equation}\label{lem-os.2}
C\le  \|\nabla u_\sigma ^k\|_{L^2(B_\rho\cap \Omega)}+\rho^{-1}k^{-1}+k^{-1}, \quad k\ge 2\mathfrak{c}/\rho,
\end{equation}
where $B_\rho=B(x_0,\rho)$ and $\mathfrak{c}$ is as in Lemma \ref{lemma-os.1}.

\end{lemma}

\begin{proof}
In this proof $C=C(n,\Omega ,\kappa)>0$ denotes a generic constant. According to Lemma \ref{lemma-a1} in Appendix \ref{appendixB}, the map 
\[
w\in H^1(\Omega) \mapsto \|\nabla w\|_{L^2(\Omega)}+\|w\|_{H^{-1/2}(\Gamma)}
\]
defines a norm, which is equivalent to the usual norm of $H^{1}(\Omega)$. Hence
\[
C\|u_\sigma ^k\|_{H^{1/2}(\Gamma)}\le \|\nabla u_\sigma ^k\|_{L^2(\Omega)}+\|u_\sigma^k\|_{H^{-1/2}(\Gamma)}= \|\nabla u_\sigma ^k\|_{L^2(\Omega)}+\|\psi_k\|_{H^{-1/2}(\Gamma)}.
\]
Using again two-sided inequality of Lemma \ref{lemma-os.1}, both for $s=-1/2$ and $s=1/2$, in order to get
\[
C\le  \|\nabla u_\sigma ^k\|_{L^2(\Omega)}+k^{-1}.
\]
This and \eqref{prop-os.1} imply 
\[
C\le  \|\nabla u_\sigma ^k\|_{L^2(B_\rho\cap \Omega)}+\rho^{-1}k^{-1}+k^{-1}, \quad k\ge 2\mathfrak{c}/\rho,
\]
as expected.
\end{proof}

Define, where $\sigma \in \Sigma$,
\[
Q_\sigma (u)=\int_\Omega \sigma |\nabla u|^2dx,\quad u\in H^1(\Omega ),
\]
and
\[
K(f)=\{u\in H^1(\Omega );\; u_{|\Gamma}=f\},\quad f\in H^{1/2}(\Gamma).
\]
We recall that
\[
Q_\sigma( u_\sigma (f))=\min\{Q_\sigma(u);\; u\in K(f)\}
\]
(e.g \cite[Section 3 in page 135]{Ka}).

Set, for $\sigma_j\in \Sigma$, $u_j^k=u_{\sigma_j}^k$ and $Q_j=Q_{\sigma_j}$, $j=1,2$. Let $x_0\in \Gamma$ so that $|\sigma (x_0)|=\|\sigma\|_{C(\Gamma)}$, where $\sigma=\sigma_1-\sigma_2$. Without loss of generality we may assume that $|\sigma (x_0)|=\sigma (x_0)$. As we have seen above
\[
\|\sigma\|_{C(\Gamma)}\le \sigma (x)+2\kappa |x-x_0|^\alpha,\quad x\in \Omega,
\] 
that we rewrite in the following form
\[
\|\sigma\|_{C(\Gamma)}+\sigma_2 (x)\le \sigma_1 (x)+2\kappa |x-x_0|^\alpha ,\quad x\in \Omega.
\] 
Hence, where $0<\rho<\mathrm{diam}(\Omega)$,
\begin{align*}
\|\nabla u_1^k\|^2_{B(x_0,\rho)\cap \Omega}\|\sigma\|_{C(\Gamma)}&+\int_{B(x_0,\rho)\cap \Omega}\sigma_2|\nabla u_1^k|^2dx
\\
&\le \int_{B(x_0,\rho)\cap \Omega}\sigma_1|\nabla u_1^k|^2dx+\rho ^\alpha \|\nabla u_1^k\|^2_{B(x_0,\rho)\cap \Omega}.
\end{align*}
That is we have
\begin{align*}
\|\nabla u_1^k\|^2_{B(x_0,\rho)\cap \Omega}\|\sigma\|_{C(\Gamma)}&+Q_2 (u_1^k)-\int_{\Omega\setminus B(x_0,\rho)}\sigma_2|\nabla u_1|^2dx
\\
&\le Q_1(u_1^k)-\int_{\Omega\setminus B(x_0,\rho)}\sigma_1|\nabla u_1|^2dx+\rho ^\alpha \|\nabla u_1^k\|^2_{B(x_0,\rho)\cap \Omega}.
\end{align*}
But $Q_2 (u_2^k)\le Q_2 (u_1^k)$. Whence
\begin{align*}
\|\nabla u_1^k\|^2_{B(x_0,\rho)\cap \Omega}\|\sigma\|_{C(\Gamma)}&+Q_2 (u_2^k)-\int_{\Omega\setminus B(x_0,\rho)}\sigma_2|\nabla u_1|^2dx
\\
&\le Q_1(u_1^k)-\int_{\Omega\setminus B(x_0,\rho)}\sigma_1|\nabla u_1|^2dx+\rho ^\alpha\|\nabla u_1^k\|^2_{B(x_0,\rho)\cap \Omega}.
\end{align*}
On the other hand, we know that
\[
Q_j (u_j^k)=\langle \Lambda_j(\psi_k),\psi_k\rangle_{1/2},\quad j=1,2.
\]
 Note that this identity yields
\[
\kappa^{-1}\|\nabla u_j^k\|_{L^2(\Omega)}^2 \le \|\Lambda_j\|\le C,\quad j=1,2.
\]
In consequence,
\begin{align*}
\|\nabla u_1^k&\|^2_{B(x_0,\rho)\cap \Omega}\|\sigma\|_{C(\Gamma)}-\int_{\Omega\setminus B(x_0,\rho)}\sigma_2|\nabla u_1|^2dx
\\
&\le \langle (\Lambda_1-\Lambda_2)(\psi_k),\psi_k\rangle_{1/2} -\int_{\Omega\setminus B(x_0,\rho)}\sigma_1|\nabla u_1|^2dx+\rho ^\alpha\|\nabla u_1^k\|^2_{B(x_0,\rho)\cap \Omega} .
\end{align*}
In light of Proposition \ref{proposition-os.1} and Lemma \ref{lemma-os.2}, we find
\[
C\|\sigma\|_{C(\Gamma)}\le \|\Lambda_1-\Lambda_2\|+\rho^{-2}k^{-2}+k^{-2}+\rho^\alpha, \quad k\ge 2\mathfrak{c}/\rho.
\]
Making first $k$ converging to $\infty$ and then $\rho$ tending to $0$ in order to obtain
\[
C\|\sigma\|_{C(\Gamma)}\le \|\Lambda_1-\Lambda_2\|.
\]
In other words we proved \eqref{thm-i1.1} of Theorem \ref{theorem-i1}.

Next, we prove \eqref{thm-i1.2} of Theorem \ref{theorem-i1} in which the exponent $\alpha/(\alpha+1)$ is substituted by $\alpha/[2(1+\alpha)]$. To this end, we
denote, where $\sigma \in \Sigma$, by $\lambda_\sigma^1=\lambda_\sigma^1(n,\Omega,\sigma)$ the first eigenvalue of the unbounded operator $-\mathrm{div}(\sigma \nabla \cdot)$ with domain $H_0^1(\Omega)\cap H^2(\Omega)$. We can associate to this eigenvalue a unique eigenfunction $\varphi_\sigma^1$ that satisfies
\[
0<\varphi_\sigma ^1\quad \mathrm{in}\; \Omega \quad \mathrm{and}\quad \|\varphi_\sigma^1\|_{L^2(\Omega)}=1.
\]

\begin{lemma}\label{lemma-os.3}
There exist $C=C(n,\Omega ,\kappa )>1$, $\varrho_0=\varrho_0(n,\Omega ,\kappa )\le \dot{\varrho}$ and $0<\beta =\beta (n)<1$ so that, for any $\sigma \in \Sigma$, we have $\varphi_\sigma^1\in C^2(\Omega)\cap C^{1,\beta}(\overline{\Omega})$ and
\[
C^{-1}\mathrm{dist}(x,\Gamma)\le \varphi_\sigma^1(x)\le C\mathrm{dist}(x,\Gamma),\quad x\in \Omega_{\varrho_0} .
\]
\end{lemma}

\begin{proof}
In this proof $C=C(n,\Omega ,\kappa )>1$ denotes a generic constant. By modifying slightly the proof of \cite[Theorem 2.2]{CTX}, we find $0<\beta =\beta (n)<1$ so that $\varphi_\sigma^1\in H^2(\Omega)\cap C^{1,\beta}(\overline{\Omega})$ and
\begin{equation}\label{os6}
\|\varphi_\sigma^1\|_{H^2(\Omega)\cap C^{1,\beta}(\overline{\Omega})}\le C.
\end{equation}

Fix $y\in \Omega$ and pick $r>0$ so that $B(y,2r)\Subset \Omega$. Let $\chi \in C_0^\infty(B(y,2r))$ satisfying, $0\le \chi \le 1$ and $\chi=1$ in a neighborhood of $B(y,r)$. Then a straightforward computations show that
\[
\mathrm{div}(\sigma \nabla (\chi \varphi_\sigma ^1))=-\lambda_\sigma ^1\chi \varphi_\sigma ^1+\sigma \nabla \chi \cdot \nabla \varphi_\sigma^1+\mathrm{div}(\sigma \varphi_\sigma^1\nabla \chi)\in C^\beta (\overline{B}(y,2r)).
\]
We get, by applying the usual H\"older regularity that $\chi \varphi_\sigma ^1\in C^{2,\min(\alpha ,\beta)}(\overline{B}(y,2r))$ (e.g. \cite[Theorem 6.8 in page 100]{GiT}). We deduce that we have in particular $\varphi_\sigma^1\in C^2(B(y,r))$. Since $y\in \Omega$ was fixed arbitrarily, we conclude that $\varphi_\sigma^1\in C^2(\Omega)$. 

Now as $\varphi_\sigma^1\in C^2(\Omega)\cap C^{1,\beta}(\overline{\Omega})$ we can apply \cite[Theorem 1.1]{CTX} in order to get
\begin{equation}\label{os7}
\|\varphi_\sigma^1\|_{L^1(\Omega)}\le -C\partial_\nu \varphi_\sigma^1(x),\quad x\in \Gamma.
\end{equation}
On the other hand, we have, in light of \eqref{os6},
\[
1=\|\varphi_\sigma ^1\|_{L^2(\Omega)}\le C\|\varphi_\sigma ^1\|_{L^1(\Omega )}.
\]
This, together with \eqref{os6} and \eqref{os7}, imply
\begin{equation}\label{os8}
C^{-1}\le -\partial_\nu \varphi_\sigma^1(x)\le C,\quad x\in \Gamma .
\end{equation}

We have, for any $x\in \Omega_{\dot{\varrho}}$,
\begin{align*}
\varphi_\sigma^1(x)&=\varphi_\sigma^1(x)-\varphi_\sigma^1(\tilde{x})=-|x-\tilde{x}|\int_0^1\nabla \varphi_\sigma^1(\tilde{x}-t|x-\tilde{x}|\nu (\tilde{x}))\cdot \nu(\tilde{x})dt
\\
&=-|x-\tilde{x}|\partial_\nu \varphi_\sigma^1(\tilde{x})
\\
&\qquad +|x-\tilde{x}|\int_0^1[\nabla \varphi_\sigma^1(\tilde{x})- \nabla \varphi_\sigma^1(\tilde{x}-t|x-\tilde{x}|\nu (\tilde{x}))]\cdot \nu(\tilde{x})dt,
\end{align*}
where $\tilde{x}=\mathfrak{p}(x)$. In light of \eqref{os6}, we obtain
\[
|x-\tilde{x}|(-\partial_\nu \varphi_\sigma^1(\tilde{x})-C|x-\tilde{x}|^\beta)\le \varphi_\sigma^1(x)\le |x-\tilde{x}|(-\partial_\nu \varphi_\sigma^1(\tilde{x})+C|x-\tilde{x}|^\beta).
\]
We derive, by using \eqref{os8}, that there exists $\varrho_0=\varrho_0(n,\Omega ,\kappa )\le \dot{\varrho}$ so that
\[
C^{-1}|x-\tilde{x}|\le \varphi_\sigma^1(x)\le C|x-\tilde{x}|,\quad x\in \Omega_{\varrho_0}.
\]
In other words, we proved that
\[
C^{-1}\mathrm{dist}(x,\Gamma)\le \varphi_\sigma^1(x)\le C\mathrm{dist}(x,\Gamma),\quad x\in \Omega_{\varrho_0},
\]
as expected
\end{proof}

In the sequel $\varrho_0$ is as in Lemma \ref{lemma-os.3}.

\begin{lemma}\label{lemma-os.4}
Let $x_0\in \Gamma$ and $0<\rho <\varrho_0$. There exist two constants $C_j=C_j(n,\Omega ,\kappa)>0$, $j=1,2$,  so that  we have
\begin{equation}\label{lem-os.4}
\int_{\Omega\cap B_\rho} \mathrm{dist}(x,\Gamma)|\nabla u_\sigma^k|^2dx\ge C_0k^{-1} -C_1(k^{-2}+\rho^{-2}k^{-2}),\quad k\ge 2\mathfrak{c}/\rho,
\end{equation}
where $B_\rho=B(x_0,\rho)$ and $\mathfrak{c}$ is as in Lemma \ref{lemma-os.1}.
\end{lemma}

\begin{proof}
Let $x_0\in \Gamma$ and $0<\rho <\varrho_0$ and set $k_\rho=2\mathfrak{c}/\rho$, $\mathfrak{c}$ is as in Lemma \ref{lemma-os.1}. We have seen above that $\mathrm{supp}(\psi_k)\subset B(x_0,\rho)$ for each $k\ge k_\rho$. Pick $\sigma \in \Sigma$ and set $u^k=u_\sigma ^k$, $\lambda^1=\lambda_\sigma^1$ and $\varphi^1=\varphi_\sigma^1$. 

In the rest of this proof we assume that $k\ge k_\rho$. Also, $C=C(n,\Omega ,\kappa)>0$ and $C_j=C_j(n,\Omega ,\kappa)>0$, $j=1,2$, denote generic constants. Taking into account that $\varphi^1u^k\in H_0^1(\Omega)$, we obtain 
\[
0=-\int_\Omega \mathrm{div}(\sigma \nabla u^k)\varphi^1u^kdx=\int_\Omega \sigma\varphi^1 |\nabla u^k|^2dx+\int_\Omega \sigma u^k\nabla u^k\cdot \nabla \varphi^1dx.
\]
But 
\begin{align*}
\int_\Omega \sigma u^k\nabla u^k\cdot \nabla \varphi^1dx&=\frac{1}{2}\int_\Omega \sigma \nabla (u^k)^2\cdot \nabla \varphi^1dx
\\
&=\frac{\lambda^1}{2}\int_\Omega \varphi^1(u^k)^2dx+\frac{1}{2}\int_\Gamma \sigma(\psi_k)^2\partial_\nu \varphi^1dS(x).
\end{align*}
Hence
\begin{equation}\label{os9}
\int_\Omega \sigma\varphi^1 |\nabla u^k|^2dx =-\frac{\lambda^1}{2}\int_\Omega \varphi^1(u^k)^2dx-\frac{1}{2}\int_\Gamma\sigma (\psi_k)^2\partial_\nu \varphi^1dS(x).
\end{equation}
Using that $\|\psi_k\|_{L^2(\Gamma)}\ge Ck^{-1/2}$ and $-\partial_\nu \varphi^1\ge C$ we get
\[
-\int_\Gamma\sigma (\psi_k)^2\partial_\nu \varphi^1dS(x)\ge Ck^{-1}.
\]
Thus we have, in light of \eqref{os9},
\begin{equation}\label{os10}
\int_\Omega \sigma\varphi^1 |\nabla u^k|^2dx \ge -\frac{\lambda^1}{2}\int_\Omega (\varphi^1u^k)u^kdx+Ck^{-1}.
\end{equation}

Denote by $u\in H_0^1(\Omega)\cap H^2(\Omega)$ the solution of the equation
\[
-\mathrm{div}(\sigma \nabla u)=\varphi^1u^k\quad \mathrm{in}\; \Omega.
\]
Then
\begin{align*}
\|(\varphi ^1)^{1/2}u^k\|_{L^2(\Omega)}^2=\int_\Omega (\varphi^1u^k)u^kdx&=-\int_\Omega \mathrm{div}(\sigma \nabla u)u^k
\\
&=-\int_\Gamma \sigma \partial_\nu uu^kdS(x).
\\
&\le C\|\partial_\nu u\|_{H^{1/2}(\Gamma)}\|u^k\|_{H^{-1/2}(\Gamma)}
\\
&\le Ck^{-1}\|\partial_\nu u \|_{H^{1/2}(\Gamma)}.
\end{align*}
The usual $H^2$ a priori estimate yields $\|\partial_\nu u \|_{H^{1/2}(\Gamma)}\le C\|\varphi ^1u^k\|_{L^2(\Omega)}$ (e.g. \cite[Theorem 8.53 in page 326]{RR} and its proof). Whence 
\[
\|\partial_\nu u \|_{H^{1/2}(\Gamma)}\le C\|(\varphi ^1)^{1/2}u^k\|_{L^2(\Omega)}. 
\]
Therefore
\[
\int_\Omega (\varphi^1u^k)u^kdx\le Ck^{-2}.
\]
This in \eqref{os10} gives
\[
\int_\Omega \varphi^1 |\nabla u^k|^2dx\ge C_0k^{-1} -C_1k^{-2}.
\]
We have also from \eqref{prop-os.1}
\[
\|\nabla u_\sigma^k\|_{H^1(\Omega\setminus \overline{B}_\rho)}\le C\rho^{-1}k^{-1}.
\]
In consequence, 
\[
\int_{B_\rho\cap \Omega} \varphi^1 |\nabla u^k|^2dx\ge C_0k^{-1} -C_1(k^{-2}+\rho^{-2}k^{-2}).
\]
We end up getting the expected inequality by applying Lemma \ref{lemma-os.3}.
\end{proof}

Assume that
\[
\|\partial_\nu \sigma\|_{C(\Gamma)}:=\epsilon >0.
\]
Let $x_0\in \Gamma$ so that $|\partial_\nu\sigma (x_0)|= \|\partial_\nu \sigma\|_{C(\Gamma)}$. We have, for $x\in \Omega\cap B(x_0,\rho)$,
\[
|\nabla \sigma (x)\cdot\nu (\tilde{x})|\ge |\nabla \sigma (x_0)\cdot \nu(\tilde{x})|-2\kappa |x-x_0|^\alpha,
\]
where we set $\tilde{x}=\mathfrak{p}(x)$. But
\[
|\nabla \sigma (x_0)\cdot \nu(\tilde{x})| \ge |\partial_\nu \sigma (x_0)|-2\kappa |\nu(\tilde{x})-\nu(x_0)|.
\]
On the other hand, as $\Omega$ is $C^{1,1}$, there exists $c=(n,\Omega,\alpha)>0$ so that
\[
|\nu(\tilde{x})-\nu(x_0)|\le c|x_0-\tilde{x}|^\alpha\le 2c|x-x_0|^\alpha.
\]
Whence
\[
|\nabla \sigma (x)\cdot\nu (\tilde{x})|\ge |\partial_\nu \sigma (x_0)|-2\kappa(1+2c) |x-x_0|^\alpha.
\]
If $\rho_0=\min ([\epsilon/(4\varkappa(1+2c) )]^{1/\alpha},\varrho_0)$, we obtain, for each $0<\rho\le \rho_0$,
\[
|\nabla \sigma (x)\cdot\nu (\tilde{x})|\ge |\partial_\nu \sigma (x_0)|/2,\quad x\in B(x_0,\rho).
\]
Let $x\in B(x_0,\rho/2)\cap \Omega$. Then, by Proposition \ref{gproposition}, we have 
\[
x_t=\tilde{x}-t|x-\tilde{x}|\nu(\tilde{x})\in \Omega_{\dot{\varrho}}\cap B(x_0,\rho),\quad 0< t\le 1,
\]
and
\[
\tilde{x}_t=\mathfrak{p}(\tilde{x}_t)=\tilde{x}.
\]
In consequence,
\[
|\nabla \sigma (x_t)\cdot\nu (\tilde{x})|\ge |\partial_\nu \sigma (x_0)|/2.
\]
In light of  the mean value theorem, there exists $0<t_0<1$ so that 
\[
|\sigma (x)-\sigma (\tilde{x})|=|\nabla \sigma (x_{t_0})\cdot \nu(\tilde{x})||x-\tilde{x}|.
\]
Whence
\[
|\sigma (x)-\sigma (\tilde{x})|=|\nabla \sigma (x_{t_0})\cdot \nu(\tilde{x})||x-\tilde{x}|\ge \mathrm{dist}(x,\Gamma)\|\partial_\nu \sigma\|_{C(\Gamma)}/2.
\]
Without loss of generality, we assume
\[
\mathrm{dist}(x,\Gamma)\|\partial_\nu \sigma\|_{C(\Gamma)}/2\le \sigma (x)-\sigma (\tilde{x}), \quad x\in B(x_0,\rho/2).
\]

We can proceed similarly as above in order to get
\[
Ck^{-1}\|\partial_\nu \sigma \|_{C(\Gamma)}\le \|\Lambda_1-\Lambda_2\|+k^{-2}+\rho^{-2}k^{-2},\quad k\ge 2\mathfrak{c}/\rho,\; 0<\rho \le \rho_0,
\]
where we used that $\|\sigma\|\le C\|\Lambda_1-\Lambda_2\|$. That is we have
\[
C\|\partial_\nu \sigma \|_{C(\Gamma)}\le k\|\Lambda_1-\Lambda_2\|+(1+\rho^{-2})k^{-1},\quad k\ge 2\mathfrak{c}/\rho,\; 0<\rho \le \rho_0.
\]
In this inequality we take $k$ of the form $k=[t+1]$ (the entire part of $t+1$), with $t\in \mathbb{R}$ satisfying $t\ge 2$. We find, by taking into account that $t\le k\le 2t$,
\[
C\|\partial_\nu \sigma \|_{C(\Gamma)}\le t\|\Lambda_1-\Lambda_2\|+(1+\rho^{-2})t^{-1},\quad t\ge 2\mathfrak{c}/\rho,\; 0<\rho \le \rho_0.
\]
It is clear that if $\|\Lambda_1-\Lambda_2\|=0$ then we get $\|\partial_\nu \sigma \|_{C(\Gamma)}=0$ by passing to the limit when $t\rightarrow \infty$ in
\[
C\|\partial_\nu \sigma \|_{C(\Gamma)}\le (1+\rho^{-2})t^{-1},\quad t\ge 2\mathfrak{c}/\rho,\; 0<\rho \le \rho_0.
\] 
But this is impossible since we assumed that $\|\partial_\nu \sigma \|_{C(\Gamma)}\ne 0$. We then choose $t$ in such a way that
\[
t\|\Lambda_1-\Lambda_2\|=(1+\rho^{-2})t^{-1}.
\]
That is
\[
t^2=(1+\rho^{-2})\|\Lambda_1-\Lambda_2\|^{-1}.
\]
This choice is possible  whenever
\[
(1+\rho^{-2})\|\Lambda_1-\Lambda_2\|^{-1}\ge 4\mathfrak{c}^2\rho^{-2},
\]
which is equivalent to the following inequality
\[
(1+\rho^2)\ge 4\mathfrak{c}^2\|\Lambda_1-\Lambda_2\|.
\]
This condition is satisfied  for instance  if
\[
4\mathfrak{c}^2\|\Lambda_1-\Lambda_2\|\le 1.
\]
Under this condition we obtain
\begin{equation}\label{nd1}
C\|\partial_\nu \sigma \|_{C(\Gamma)}\le (1+\rho^{-1})\|\Lambda_1-\Lambda_2\|^{1/2}.
\end{equation}
When $\rho_0=\varrho_0$ then the last inequality yields in a straightforward manner, by taking $\rho=\varrho_0$, that
\[
C\|\partial_\nu \sigma \|_{C(\Gamma)}\le \|\Lambda_1-\Lambda_2\|^{1/2},
\]
Otherwise, we have $\rho_0=[\epsilon/(4\varkappa(1+c) )]^{1/\alpha}=\tilde{c}\epsilon^{1/\alpha}$. We get, by taking $\rho=\rho_0$ in \eqref{nd1},
\[
C\|\partial_\nu \sigma \|_{C(\Gamma)}\le (1+\|\partial_\nu \sigma \|_{C(\Gamma)}^{-1/\alpha})\|\Lambda_1-\Lambda_2\|^{1/2}.
\]
Hence
\[
C\|\partial_\nu \sigma \|_{C(\Gamma)}^{1+1/\alpha}\le \|\Lambda_1-\Lambda_2\|^{1/2}.
\]
In other words, we have
\[
C\|\partial_\nu \sigma \|_{C(\Gamma)}\le \|\Lambda_1-\Lambda_2\|^{\alpha/[2(1+\alpha)]}.
\]
This estimate is obviously satisfied if $4\mathfrak{c}^2\|\Lambda_1-\Lambda_2\|\ge 1$. Hence the expected inequality follows.

The results of this section improve and complete those of \cite{Ka}.

\appendix

\section{}\label{appendixA}

We will use the following lemma in the proof of Proposition \ref{gproposition}.
\begin{lemma}\label{glemma}
Assume that $\Omega$ is of class $C^{1,1}$. If $x_0\in \Gamma$ then we find $\varrho=\varrho(x_0)>0$ so that :
\\
$\mathrm{(i)}$ for any $x\in \Omega \cap B(x_0,\varrho)$, there exists a unique $\mathfrak{p}(x)\in \Gamma$ such that
\[
|x-\mathfrak{p}(x)|=\mathrm{dist}(x,\Gamma)\quad \mathrm{and}\quad x= \mathfrak{p}(x)-|x-\mathfrak{p}(x)|\nu(\mathfrak{p}(x)),
\]
$\mathrm{(ii)}$ when $x\in \Omega \cap B(x_0,\varrho/2)$, we have $x_t=\mathfrak{p}(x)-t|x-\mathfrak{p}(x)|\nu(\mathfrak{p}(x))\in \Omega \cap B(x_0,\varrho)$, $t\in ]0,1]$, and
\[
\mathfrak{p}(x_t)=\mathfrak{p}(x)\quad \mbox{and}\quad |x_t-\mathfrak{p}(x)|=t\mathrm{dist}(x,\Gamma).
\]
\end{lemma}

\begin{proof}
$\mathrm{(i)}$ follows readily from \cite[Theorem 4.3 in page 219]{DZ} and \cite[formula (3.4)  in page 214]{DZ}. While $\mathrm{(ii)}$ is a consequence of \cite[Theorem 4.4 in page 224]{DZ}.
\end{proof}

Observe that Lemma \ref{glemma} is no longer valid for domains with less regularity than $C^{1,1}$ (e.g \cite[Example 4.1, page 222]{DZ}).

\begin{proof}[Proof of Proposition \ref{gproposition}]
We keep the notations of the preceding lemma. By compactness of $\Gamma$, we find $x_0^1,\ldots ,x_0^p$ in $\Gamma$ so that
\begin{equation}\label{glem1}
\Gamma \subset \bigcup_{j=1}^pB(x_0^j,\varrho(x_0^j)/2).
\end{equation}
We claim that there exists $\dot{\varrho}>0$ so that
\[
\Omega \setminus \bigcup_{j=1}^pB(x_0^j,\varrho(x_0^j)/2)\subset \{x\in \Omega ;\; \mathrm{dist}(x,\Gamma)>\dot{\varrho}\}.
\]
We proceed by contradiction. So, if this is not the case, we find a sequence $(x_k)$ in $\Omega \setminus \bigcup_{j=1}^p\overline{B}(x_0^j,\varrho(x_0^j)/2)$ satisfying $\mathrm{dist}(x_k,\Gamma)\rightarrow 0$ when $k$ goes to $\infty$. Subtracting a subsequence if necessary, we may assume that $x_k$ converge to $\overline{x}\in \overline{\Omega}$. The continuity of the distance function yields $\mathrm{dist}(\overline{x},\Gamma)=0$ and hence $\overline{x}\in \Gamma$. In light of \eqref{glem1}, $\overline{x}\in B(x_0^j,\varrho(x_0^j)/2)$, for some $1\le j\le p$. Whence $x_k \in B(x_0^j,\varrho(x_0^j)/2)$, when $k$ sufficiently large, which leads to the expected contradiction. In other words, we proved
\[
\Omega_{\dot{\varrho}}\subset \bigcup_{j=1}^pB(x_0^j,\varrho(x_0^j)/2).
\]
We complete the proof by using Lemma \ref{glemma}.
\end{proof}

\section{}\label{appendixB}

We prove the following lemma.

\begin{lemma}\label{lemma-a1}
We have
\begin{equation}\label{ap1}
\inf\{ \|\nabla w\|_{L^2(\Omega)}+\|w\|_{H^{-1/2}(\Gamma)};\; w\in H^1(\Omega),\; \|w\|_{L^2(\Omega)}=1\}>0.
\end{equation}
\end{lemma}

\begin{proof}
We proceed by contradiction. If \eqref{ap1} does not hold then we find a sequence $(w_k)_{k\ge 1}\in H^1(\Omega )$ satisfying $\|w_k\|_{L^2(\Omega)}=1$ and
\begin{equation}\label{ap2}
\|\nabla w_k\|_{L^2(\Omega)}+\|w_k\|_{H^{-1/2}(\Gamma)}\le 1/k,\quad k\ge 1.
\end{equation}
In particular, $(w_k)$ is bounded in $H^1(\Omega)$. Subtracting a subsequence if necessary, we may assume that $(w_k)$ converges weakly in $H^1(\Omega)$ and strongly in $L^2(\Omega)$ to $w\in H^1(\Omega)$. As $h\in H^1(\Omega)\mapsto h_{|\Gamma}\in H^{-1/2}(\Gamma)$ is linear and continuous, it is also continuous when $H^1(\Omega)$ and  $H^{-1/2}(\Gamma)$ are endowed with their weak topology. Thus $w_k{_{|\Gamma}}$ converges weakly to $w_{|\Gamma}$ in $H^{-1/2}(\Gamma)$. Therefore, using that any norm is lower semi-continuous for the weak topology, we obtain from \eqref{ap2}
\[
\|\nabla w\|_{L^2(\Omega)}+\|w\|_{H^{-1/2}(\Gamma)}\le \liminf_k\|\nabla w_k\|_{L^2(\Omega)}+\liminf_k\|w_k\|_{H^{-1/2}(\Gamma)}=0
\]
and hence $w=0$. But this contradicts the fact that
\[
1=\lim_k\|w_k\|_{L^2(\Omega)}=\|w\|_{L^2(\Omega)}.
\]
The proof is then complete.
\end{proof}

\end{document}